\documentclass[11pt]{article}

\usepackage{jmlr2e}

\usepackage{amssymb,amsmath,amsfonts, mathtools,color}

\usepackage{mathrsfs}

\usepackage[normalem]{ulem}

\usepackage[title,titletoc,header]{appendix}
\usepackage{graphicx,subfig}

\usepackage{paralist}

\usepackage{tikz}
\usetikzlibrary{arrows,positioning,shapes.geometric}

\usepackage{algorithm,algorithmic}

\graphicspath{ {./figures/} }

\usepackage{indentfirst}
\usepackage{multicol}
\usepackage{booktabs}
\usepackage{url}
\usepackage[outdir=./]{epstopdf} 

\usepackage[shortlabels]{enumitem}

\usepackage{hyperref}
\hypersetup{
    colorlinks=true, 
    linktoc=all,     
    citecolor=black,
    linkcolor=black,  
}
\numberwithin{equation}{section}

\newtheorem{Theorem}{Theorem}[section]
\newtheorem{Lemma}[Theorem]{Lemma}
\newtheorem{Proposition}[Theorem]{Proposition}
\newtheorem{Assumption}{H.\!\!}

\newtheorem{Definition}{Definition}[section]

\newtheorem{Remark}{Remark}[section]

 \def\p{\partial} 
\def\to{\rightarrow}
     
\def\Om{\Omega}   
 
\newcommand{\q}{\quad}   \newcommand{\qq}{\qquad}

  \def\fa{\forall}
  \def\a{\alpha} 
 
\def\eps{\varepsilon}

 \def\t{\times}

\def\cF{\mathcal{F}}

\def\cH{\mathcal{H}}

\def\cO{\mathcal{O}}

\def\cS{\mathcal{S}}

\def\d{{\mathrm{d}}}

\def\sE{{\mathbb{E}}}
\def\sF{{\mathbb{F}}}

\def\sN{{\mathbb{N}}}
\def\sP{\mathbb{P}}

\def\sR{{\mathbb R}}
\def\sS{{\mathbb{S}}}
\def\sX{{\mathbb{X}}}
\def\sY{{\mathbb{Y}}}

\def\trans{\mathsf{T}}

\newcommand{\tr}{\textnormal{tr}}

\DeclareMathOperator*{\argmin}{arg\,min}

\newcommand{\lc}
{\mathrel{\raise2pt\hbox{${\mathop<\limits_{\raise1pt\hbox
{\mbox{$\sim$}}}}$}}}

\newcommand{\gc}
{\mathrel{\raise2pt\hbox{${\mathop>\limits_{\raise1pt\hbox{\mbox{$\sim$}}}}$}}}

\newcommand{\ec}
{\mathrel{\raise2pt\hbox{${\mathop=\limits_{\raise1pt\hbox{\mbox{$\sim$}}}}$}}}

\def\bb{\begin{equation}} \def\ee{\end{equation}}
\def\bbn{\begin{equation*}} \def\een{\end{equation*}}

\def\beqn{\begin{eqnarray}}  \def\eqn{\end{eqnarray}}

\def\beqnx{\begin{eqnarray*}} \def\eqnx{\end{eqnarray*}}

\def\bn{\begin{enumerate}} \def\en{\end{enumerate}}

\def\bd{\begin{description}} \def\ed{\end{description}}



\ShortHeadings{
Logarithmic Regret  for
Episodic 
Continuous-Time
LQ-RL
}{Basei, Guo, Hu and Zhang}
\firstpageno{1}

\begin{document}

\title{
Logarithmic Regret  for
Episodic 
Continuous-Time
Linear-Quadratic Reinforcement Learning
over a Finite-Time  Horizon }

       
\author{\name Matteo Basei \email  matteo.basei@edf.fr\\
\addr EDF R\&D and FIME, Paris, France. \\
\AND
\name Xin Guo \email  xinguo@berkeley.edu\\
\addr Department of Industrial Engineering and Operations Research\\
University of California, Berkeley, USA.  \\
\AND
\name Anran Hu \email  anran\_hu@berkeley.edu\\
\addr Department of Industrial Engineering and Operations Research\\
 University of California, Berkeley, USA. \\
\AND
\name Yufei Zhang \email  yufei.zhang@maths.ox.ac.uk\\
\addr Mathematical Institute\\
University of Oxford, UK. 
}

\editor{}

%
%
%
\date{}
\maketitle

\begin{abstract}%
We study finite-time horizon continuous-time linear-quadratic reinforcement learning problems
in an episodic setting,
where both  the state and control coefficients are unknown to the controller. 
We first propose a least-squares  algorithm based on continuous-time observations and controls,
 and establish a logarithmic 
regret bound of magnitude  $\cO((\ln M)
(\ln\ln M) )$, with $M$ being the number of learning episodes.
The analysis consists of two components:  perturbation analysis, which exploits the regularity and  robustness
of the associated Riccati differential equation;
and parameter
estimation error, which relies on  sub-exponential properties of  continuous-time least-squares estimators.
We further  propose a practically implementable least-squares  algorithm based on discrete-time observations and piecewise constant controls,
 which achieves similar logarithmic  regret 
 with an additional  term 
 depending explicitly on  the  time stepsizes used in the algorithm.
\end{abstract}
%
%
%


\medskip

\section{Introduction}\label{sec:intro}
Reinforcement learning (RL) for  linear quadratic (LQ) control problems  has been one of the most active areas  for both the control and the reinforcement learning communities. Over the last few decades,  significant progresses have been made in the discrete-time setting. 
\subsection{Discrete-time RL}
\label{sec:discrete_time_rl}
{\color{black}
In the area  of adaptive control  
with unknown dynamics parameters, the goal  is to find optimal stationary policy that stabilizes the unknown dynamics and minimizes the long term average cost (\cite{ioannou2006adaptive,landau2011adaptive}).
For an infinite-time horizon LQ system, it has been shown that persistent excitation conditions \cite{green1986persistence} are critical to the parameter identification. 
Meanwhile, algorithms with asymptotic convergence in both the parameter estimation and the optimal control have been developed in \cite{goodwin1981discrete}, \cite{kumar1983optimal} and \cite{campi1998adaptive}: the first one assumes that costs only depend on state variables and the other two consider both state and control costs and use a cost-biased least-squared estimation method. 
See
\cite{faradonbeh2018finite,faradonbeh2019randomized}
 and references therein for recent developments of (randomised) adaptive control algorithms for LQ systems.
}

Following the seminal works of  
 \cite{auer2007logarithmic, auer2009near} and \cite{osband2013more}, 
 \textit{non-asymptotic} regret bound analysis for RL algorithms
has been one of the main topics, and has been developed for tabular  Markov decision problems.

The non-asymptotic analysis of adaptive LQ problem 
by \cite{abbasi2011regret} utilizes the Optimism in the Face of Uncertainty   principle
to construct a sequence of improving confidence regions for the unknown model parameters, and  solves a non-convex constrained optimization problem for   each  confidence region; their algorithm  achieves an $ \cO(\sqrt{T})$ regret bound, with $T$ being  
the  number of time steps. To reduce the computational complexity and to avoid the non-convexity issue, \cite{abeille2018improved} and \cite{ouyang2017learning} propose Thompson-sampling-based algorithms and derive $\cO(\sqrt{T})$ regret bounds in the Bayesian setting; \cite{Robust_adacont} proposes a robust adaptive control algorithm to solve a convex sub-problem in each step and achieves an $\cO(T^{2/3})$ regret bound.
 The gap between these regret bounds is removed by \cite{certainty_ben} and \cite{cohen2019learning}  via two different approaches for the same $\cO(\sqrt{T})$ frequentist regret bound.  Later, \cite{simchowitz2020naive} establishes a lower bound on the regret of order $\cO(\sqrt{d_u^2d_xT})$, where $d_u$ and $d_x$ are the dimensions of the actions and the states, and shows that a simple
variant of certainty equivalent control matches the lower bound in both $T$ and the dimensions. Similar regret bounds have also been established under different settings and assumptions, such as \cite{chen2021black} in the adversarial setting and \cite{lale2020explore} without a stabilizing controller at the early stages of agent-environment interaction. 

{\color{black}
All the analyses are in discrete-time with an infinite time horizon.
In all these  problems,  adaptive control algorithms are shown to achieve logarithmic regret bounds when additional information regarding the parameters of the system 
(often referred to as identifiability conditions)
 is available.
Indeed,
\cite{kazem2018adaptive,faradonbeh2020input} prove that certainty equivalent adaptive regulator
achieves logarithmic regret bounds
if the system parameter satisfies certain sparsity or low-rankness conditions. \cite{cassel2020logarithmic} establishes logarithmic regret bounds when either the state transition matrix is unknown, or when the state-action transition matrix is unknown and the optimal policy is non-degenerate. In partially observable linear dynamical systems, which takes linear-quadratic Gaussian problem as a special case, \cite{lale2020logarithmic} proposes an algorithm with a logarithmic regret bound, under the assumption that one has access to a set in which all controllers persistently excite the system to  approximate the optimal control. Logarithmic regret bounds in the adversarial setting with known dynamics parameters have been established in \cite{agarwal2019logarithmic,foster2020logarithmic}. 
}

\subsection{Continuous-time RL.}
Most real-world control systems, such as those in aerospace, automotive industry and robotics,
are naturally  continuous-time dynamical systems.
{\color{black}
So are their related physical tasks, such as inverted pendulum problems,
cart-pole balancing problems, and legged robot problems. 
Continuous-time finite-time horizon LQ control problems can be found in 
  portfolio optimization \cite{wang2020continuous},  algorithmic trading \cite{cartea2018algorithmic}, 
 production management of 
 exhaustible resources \cite{graber2016linear}, and biological movement systems
\cite{bailo2018optimal}.
}

Analysis for continuous-time LQ-RL and general RL problems, however,  is fairly limited.
The primary approach is to develop learning algorithms after discretizing both the time and the space spaces, and establish the convergence as discretization parameters tend to zero. For instance, \cite{munos2006policy} proposes a policy gradient
algorithm and shows the convergence of the policy gradient estimate to the true gradient.
\cite{munos1998reinforcement, munos2000study} 
design  learning algorithms
 by discretizing Bellman  equations of the underlying control problems
 and prove the  asymptotic convergence of their algorithms. 
For the  LQ
system, attentions have been mostly on algorithms designs, including the integral reinforcement learning algorithm in \cite{modares2014linear}, and the policy iteration algorithm in \cite{rizvi2018output}. Yet, very little 
 is known regarding the convergence rate or the regret bound of all these algorithms.
 Indeed, despite the natural analytical connection between LQ control and RL, the best known theoretical
work for continuous-time LQ-RL is still due to \cite{duncan1999adaptive}, where an asymptotically sublinear regret for  an ergodic model has been derived via a weighted least-squares-based estimation approach.
Nevertheless, the exact order of the regret bound has not been studied.

\paragraph{Issues and challenges from {non-asymptotic analysis}.}


{\color{black}
It is insufficient and improper to rely solely on the analysis and algorithms  for the discrete-time RL to solve the continuous-time problems. There is a mismatch between the algorithms
timescale for the former and the underlying systems timescale for the latter.
When designing algorithms that make  observations and take actions at  discrete time points,
  it is important to take the model mismatch into consideration.
For instance, the empirical studies in \cite{tallec2019making} suggest  that
 vanilla 
  $Q$-learning methods exhibit degraded performance
  as the time stepsize decreases, 
  while 
   a proper scaling of learning rates with stepsize 
  leads to more robust performance.

The questions are therefore: A) How to  quantify the precise impacts of 
the observation stepsize 
and action stepsize
on algorithm performance? B)  How to derive non-asymptotic regret analysis for learning algorithms in continuous-time LQ-RL (or general RL) system,
analogous to the discrete-time LQ-RL counterpart?
}

 There are technical reasons behind the limited   theoretical progress in the continuous-time domain for RL, including LQ-RL. 
 In addition to the known difficulty   for analyzing
 stochastic control problems, the learning component  compounds the problem complexity and poses new challenges. 
 
For instance, the counterpart in the continuous-time problem  to  the algebraic equations in \cite{certainty_ben} for the discrete-time version is the regularity and stability of the continuous-time
Riccati  equation and the regularity of feedback controls. While Riccati equation and its robustness and existence and uniqueness of optimal controls have been well studied in the control literature,
regularity of feedback controls
with respect to underlying models
is completely new for control theory and crucial for algorithm design and its robustness analysis. 
Moreover, deriving the {\it exact}
order of the regret bound requires developing new and different  techniques than  those used for the  {\it asymptotic} regret analysis in \cite{duncan1999adaptive}.
  

\paragraph{Our work and contributions.}
This paper studies  finite-time horizon continuous-time LQ-RL problems  in an episodic setting.
 \begin{itemize}[leftmargin=*]
\item

 It  first proposes a   greedy least-squares   algorithm
 based on  continuous-time observations and controls. 
 At each iteration,
the algorithm   estimates the unknown 
parameters by a regularized least-squares estimator based on observed trajectories,
then designs
linear feedback controls via the Riccati differential equation
for  the estimated model.
It identifies conditions under which the unknown state transition matrix and state-action transition matrix  are uniquely identifiable 
under the optimal policies. 
(Remark \ref{rmk:identificable}
and Proposition \ref{prop:non_degenerate_statement}).
By exploiting the identifiability of coefficients, 
this continuous-time least-squares  algorithm is shown to have a 
logarithmic  regret 
of the magnitude $\cO((\ln M)(\ln\ln M) )$, with $M$ being the number of learning episodes
(Theorem \ref{thm:regret_continuous}).
 To the best of our knowledge, this is the first  non-asymptotic logarithmic regret bound for continuous-time LQ-RL problems
with unknown   state and control coefficients.

\item 
 It  then proposes a   practically implementable  least-squares   algorithm
 based on  discrete-time observations and controls. 
 At each iteration,
 the algorithm estimates 
the unknown  parameters
by  observing  continuous-time trajectories
at discrete time points,
then designs 
 a piecewise constant 
linear feedback control  via  Riccati difference equations
for an associated discrete-time LQ-RL problem.
It shows that 
  the regret of the discrete-time least-squares  algorithm is 
of the magnitude $\cO\big((\ln M)(\ln\ln M)+\sum_{\ell=0}^{\ln M} 2^\ell \tau_\ell^{2}\big)$, where
$\tau_\ell$ is the  time stepsize used in the $(\ell+1)$-th update of model parameters
(Theorem \ref{thm:regret_discrete}).
{\color{black}
Our analysis shows that  scaling the regularization parameter of  the discrete-time least-squares estimator with respect to time stepsize is critical for a robust performance of the algorithm in different timescales
(Remark \ref{rmk:scaling}).}
 To the best of our knowledge, this is the first discrete-time algorithm with rigorous regret bound for  continuous-time LQ-RL problems.

\end{itemize}

Different from the    least-squares algorithms  for the 
 ergodic  LQ problems (see e.g., \cite{duncan1999adaptive,certainty_ben}), 
   our continuous-time least-squares algorithm constructs feedback controls via Riccati differential 
   equations instead of  
   the algebraic equations in \cite{certainty_ben}.
 Here, the regularity and stability of the continuous-time
Riccati  equation is analyzed in 
 order to establish the robustness of feedback controls.
 
Moreover, our analysis for the estimation error exploits extensively the sub-exponential tail behavior of the least-squares estimators. This probabilistic approach differs from the asymptotic sublinear regret analysis in \cite{duncan1999adaptive}; it 
establishes the exact
order of the logarithmic regret bound by the concentration inequality for the error bound.

{\color{black}
 In addition,  our   analysis
also 
 exploits an important self-exploration property of finite-time horizon
 continuous-time 
 LQ-RL problems, for which
the  time-dependent optimal feedback matrices  
ensure that  the optimal state and control processes span the entire parameter space.
This property allows us to design exploration-free learning algorithms
with  logarithmic regret bounds.
Furthermore,
we  provide explicit conditions on models that guarantees the successful identification of the unknown parameters
 with optimal feedback  policies.
This is in contrast to the   identification conditions
  for logarithmic regret bounds in discrete-time infinite-time-horizon LQ problems. 
 Our conditions apply to arbitrary finite time-horizon problems, without
 imposing sparsity or low-rankness conditions on system parameters as in 
\cite{kazem2018adaptive,faradonbeh2020input}
 or requiring these parameters to be partially known to the controller as in 
  \cite{cassel2020logarithmic,foster2020logarithmic}.
}

Finally,  our analysis 
 provides the precise parameter estimation error in terms of the sample size and time stepsize, and
 quantifies the performance gap between applying a  piecewise-constant policy from an incorrect model
and applying the optimal policy. The
misspecification error 
 scales  linearly with respect to the stepsize, and the performance gap
 depends quadratically with respect to  the time stepsize and the magnitude of  parameter perturbations. 
Our analysis is based on the first-order convergence of Riccati difference equations and a uniform sub-exponential tail bound of discrete-time least-squares estimators.

\paragraph{Notation.} 

For each  $n\in \sN$, we denote by $I=I_n$ the $n\t n$ identity matrix,
and
by $\sS^n_0$ (resp.~$\sS^n_+$) the space of symmetric positive semidefinite (resp.~definite) matrices.
We denote by $|\cdot|$ the  Euclidean norm of a given Euclidean space, 
by $\|\cdot\|_2$ the matrix norm induced by Euclidean norms,
and 
 by $A^\top$ and $\tr(A)$ the transpose  and trace of a matrix $A$, respectively.
For each $T> 0$, filtered probability space $(\Om,\cF,\sF=\{\cF_t\}_{t\in[0,T]},\sP)$ 
 satisfying the usual condition and  Euclidean space $(E,|\cdot|)$, we introduce the following spaces:
 \begin{itemize}[leftmargin=*,noitemsep,topsep=0pt]
\item 
 $C([0,T];E)$ is the space of continuous  functions
 $\phi:[0,T]\to E$ satisfying 
 $\|\phi\|_{C([0,T];E)}=\sup_{t\in [0,T]}|\phi_t|<\infty$;
 \item
 $C^1([0,T];E)$ is the space of continuously differentiable  functions
 $\phi:[0,T]\to E$ satisfying 
 $\|\phi\|_{C^1([0,T];E)}=\sup_{t\in [0,T]}(|\phi_t|+|\phi'_t|)<\infty$;
\item
 $\cS^2(E)$
  is the space of 
 $E$-valued
$\sF$-progressively  measurable c\`{a}dl\`{a}g
processes
$X: \Om\t [0,T]\to E$ 
satisfying $\|X\|_{\cS^2(E)}=\sE[\sup_{t\in [0,T]}|X_t|^2]^{1/2}<\infty$;
\item
 $\cH^2(E)$  is the space of 
   $E$-valued $\sF$-progressively measurable
 processes 
$X: \Om\t [0,T]\to E$ 
 satisfying $\|X\|_{\cH^2(E)}=\sE[\int_0^T|X_t|^2\,\d t]^{1/2}<\infty$.
\end{itemize}
For notation simplicity, we denote by $C\in [0,\infty)$ a generic constant, which depends only on the constants appearing in the assumptions and may take a different value at each occurrence.

\section{Problem formulation and main results}

\subsection{Linear-quadratic reinforcement learning problem}
In this section, we consider  the linear-quadratic reinforcement learning (LQ-RL) problem,
where the drift coefficient of the state dynamics 
is unknown to the controller.
%
%

More precisely, 
let 
 $T\in (0,\infty)$ be a given terminal time,
 $W$ be an $n$-dimensional  standard Brownian motion defined on a complete probability space $(\Omega,\mathcal{F},\mathbb{P})$, 
and  $\sF=(\cF_t)_{t\in [0,T]}$ be the  filtration
generated by $W$   augmented by the $\sP$-null sets.
Let 
$x_0\in \sR^n$ be a given initial state
and 
$(A^\star, B^\star)
\in \sR^{n\t n}\t \sR^{n\t d}$
be
fixed but unknown matrices,
 consider the following problem:
\bb\label{LQ}
\inf_{U\in \cH^2 (\sR^d)} J^{\theta^\star}(U),
\q \textnormal{with}\q 
J^{\theta^\star}(U)=
\sE\left[
\int_0^T\big((X^{\theta^\star,U}_t)^\top QX^{\theta^\star,U}_t+(U_t)^\top RU_t\big)\,\d t
\right],
\ee
where for each $U\in  \cH^2(\sR^d)$, the process $X^{\theta^\star,U}\in \cS^2(\sR^n)$ satisfies the following controlled dynamics
 associated with the parameter $\theta^\star=(A^\star, B^\star)^\top$: 
\bb\label{dynamics}
\d X_t =(A^\star X_t+B^\star U_t)\,\d t+\, \d W_t, \q t\in [0,T];
\q X_0=x_0,
\ee
with 
given matrices $Q\in \sS^n_0$ and $R\in \sS^d_+$.
Note that 
we assume the  loss functional 
\eqref{LQ} only involves 
a time homogeneous 
running cost
to allow a direct comparison with  infinite-time horizon 
RL problems (see e.g., \cite{duncan1992least}), 
but similar analysis can be performed 
if the cost functions are time inhomogeneous, 
 a terminal cost is included,
or the Brownian motion $W$ in \eqref{dynamics} is scaled by an known nonsingular diffusion matrix.

If the  parameter $\theta^\star=(A^\star, B^\star)^\top$ are known to the controller,
then \eqref{LQ}-\eqref{dynamics}
reduces  to the classical   LQ control problems.
In this case, 
it is well known that (see e.g., \cite{yong1999stochastic} and the references therein),  the optimal control $U^{\theta^\star}$ 
of \eqref{LQ}-\eqref{dynamics} 
is given in a  feedback form by
\begin{equation}
\label{K1}
U^{\theta^\star}_t=\psi^{\theta^\star}(t, X^{\theta^\star}_t), \q
\textnormal{
with $\psi^{\theta^\star}(t,x)=
K^{\theta^\star}_t x$, $\fa (t,x)\in [0,T]\t \sR^n$,}
\end{equation}
where $K^{\theta^\star}_t= -R^{-1}(B^{\star})^\top P^{\theta^\star}_t$ for all $t\in [0,T]$,
 $(P_t^{\theta^\star})_{t\in [0,T]}$ solves the Riccati equation
\begin{equation}
\label{riccati1}
\tfrac{\d}{\d t}P_t + (A^\star)^\top P_t + P_tA^\star  - P_t(B^\star R^{-1}(B^\star)^\top )P_t + Q=0, \q t \in [0,T];
\q
P_T=0,
\end{equation}
and $X^{\theta^\star}$ is the   state process governed by the following dynamics: 
\bb\label{eq:state_psi_theta}
\d X_t =(A^\star X_t+B^\star K^{\theta^\star}_t X_t)\,\d t+\, \d W_t, \q t\in [0,T];
\q X_0=x_0.
\ee

To solve  the LQ-RL problem 
 \eqref{LQ}-\eqref{dynamics}
 with unknown $\theta^\star$,
 the controller  searches for the optimal control while simultaneously learning the system, i.e., the matrices $A^\star, B^\star$.
 In an episodic (also known as reset or restart) learning framework, the controller improves her knowledge of the underlying  dynamics $X_t$ through successive learning episodes, in order to find a  control that is close to the optimal one.

Mathematically, it goes as follows. Let $M\in\mathbb{N}$ be the total number of learning episodes. In the $i$-th learning episode, $i=1,\ldots,M$, a feedback control $\psi^i$ is exercised, and  the  state process  $X^{\psi^i}$ evolves according to  the dynamics \eqref{dynamics} controlled by  the policy  $\psi^i$:
\begin{equation}\label{dynamics-n}
    \begin{split}
        \d X_t&=(A^\star X_t+B^\star \psi^i(t,X_t))\d t+\d W_t^i,
        \q t\in [0,T];
         \quad X_0=x_0. \\
    \end{split}
\end{equation}
 Here $W^i ,i =1,2,\ldots, M$ are independent $n$-dimensional Brownian motions defined on the same  probability space $(\Omega,\mathcal{F},\mathbb{P})$.
The (expected) cost of learning in the $i$-th episode is then given by
\begin{equation}
\label{J-n}
J^{\theta^\star}(U^{\psi^i}) = \mathbb{E} \left[\int_0^T
\big(
(X^{\psi^i}_t)^\top QX^{\psi^i}_t+(U^{\psi^i}_t)^\top R U^{\psi^i}_t\big)\, \d t\right],
\;
\textnormal{with $U^{\psi^i}_t\coloneqq \psi^i(t,X^{\psi^i}_t)$,  $t\in [0,T]$,}
\end{equation}
and 
the (expected) regret of learning 
up to $M \in \sN$ episodes  (with the sequence of controls $(U^{\psi^i})_{i=1}^{M}$)
is defined  as follows:
\begin{equation}\label{reg}
    R(M)=\sum_{i=1}^M \Big(J^{\theta^\star}(U^{\psi^i})-J^{\theta^\star}(U^{\theta^\star})\Big),
\end{equation}
where $J^{\theta^\star}(U^{\theta^\star})$ is the optimal cost 
of \eqref{LQ}-\eqref{dynamics}
 when $A^\star, B^\star$ are known. 
 Intuitively, the regret characterizes the cumulative loss from taking sub-optimal policies in all episodes.

In the following, we shall propose several   least-squares-based learning algorithms 
to solve \eqref{LQ}-\eqref{dynamics},
and prove that they achieve   logarithmic   regrets if  $\theta^{\star}$ is identifiable
(see Remark \ref{rmk:identificable} for details).

\subsection{Continuous-time least-squares algorithm and its regret bound}
\label{sec:conts_ls}

In this section, we consider a 
continuous-time least-squares algorithm, 
which chooses the optimal feedback control based on the current estimation of the parameter, and 
updates the parameter estimation based on the whole trajectories of the state dynamics.

More precisely, 
let  
$\theta =(A,B)^\top\in \sR^{(n+d)\t n}$
be the current estimate of the unknown parameter $\theta^\star $,
then the controller would exercise the optimal feedback control $\psi^\theta$ 
 for   \eqref{LQ}-\eqref{dynamics} 
 with $\theta^\star$ replaced by $\theta$, i.e.,
\begin{equation}
\label{K2nonoise}
 \psi^\theta(t,x)= K^\theta_t x, 
 \q   K^\theta_t\coloneqq -R^{-1} B^\top  P^\theta_t, \q \fa (t,x)\in [0,T]\t \sR^n,
\end{equation}
where  $ P^\theta$ satisfies   the Riccati equation \eqref{riccati1}
  with $\theta^\star$ replaced by $\theta$:
\begin{equation}
\label{riccati2}
\tfrac{\d}{\d t}P_t + A^\top P_t + P_tA   - P_t(B  R^{-1}B ^\top )P_t + Q=0, \q t \in [0,T];
\q
P_T=0.
\end{equation}
This leads to the  state process $X^{\psi^\theta}$ satisfying 
(cf.~\eqref{dynamics-n}):
\begin{equation}\label{lin_dyn}
   \d  X_t=
   (A^\star X_t+   B^\star \psi^\theta(t, X_t))\,\d t+\d W_t,
   \q t\in[0,T];
    \quad X_0=x_0. 
\end{equation}

We proceed to derive 
an $\ell_2$-regularized least-squares estimation for $\theta^\star$ based on sampled trajectories 
of $X^{\psi^\theta}$.
Observing from \eqref{lin_dyn} that 
 $$ Z^{\psi^\theta}_t(\d X^{\psi^\theta}_t)^\top =Z^{\psi^\theta}_t (Z^{\psi^\theta}_t)^\top \theta^\star \d t+
  Z^{\psi^\theta}_t (\d W_t)^\top,
\q 
\textnormal{
with
$Z^{\psi^\theta}_t=
\begin{psmallmatrix} X^{\psi^\theta}
\\
\psi^{\theta}(t,X^{\psi^\theta}_t)
 \end{psmallmatrix}
$ for all $t\in [0,T]$.
}
 $$
Hence the martingale property of the It\^{o} integral implies that 
\begin{equation}
\label{rewrTheta}
\theta^\star = \left(\mathbb{E}\left[\int_0^T Z^{\psi^\theta}_t (Z^{\psi^\theta}_t)^\top\, \d t\right]\right)^{-1} \mathbb{E}\left[\int_0^T Z^{\psi^\theta}_t(\d X^{\psi^\theta}_t)^\top \right],
\end{equation}
provided that $\mathbb{E}\big[\int_0^T Z^{\psi^\theta}_t (Z^{\psi^\theta}_t)^\top\, \d t\big]$ is invertible. 
This 
 suggests a practical rule to improve one's estimate $\theta$ for the true parameter
$\theta^\star$, 
by replacing the expectations in  \eqref{rewrTheta}
with  empirical averages over independent realizations.
More precisely, let 
$m\in \sN$ and 
$(X^{\psi^\theta,i}_t, \psi^\theta(t,X^{\psi^\theta,i}_t))_{t\in [0,T]}$, $i=1,\ldots,m$,
be trajectories of $m$ independent  realizations of the state and control processes,
we shall update the estimate $\theta$  by the following rule, inspired by \eqref{rewrTheta}:
\bb
\label{reg_ls}
\theta
\longleftarrow
\bigg(
\frac{1}{m}\sum_{i=1}^m
\int_0^T
Z^{\psi^\theta,i}_t (Z^{\psi^\theta,i}_t)^\trans
\,\d t
+\frac{1}{m}I
\bigg)^{-1}
\bigg(
\frac{1}{m}\sum_{i=1}^m
\int_0^T
Z^{\psi^\theta,i}_t (\d X^{\psi^\theta,i}_t)^\trans
\bigg),
\ee
where 
 $Z^{\psi^\theta,i}_t\coloneqq
\begin{psmallmatrix} X^{\psi^\theta,i}_t \\ \psi^\theta(t,X^{\psi^\theta,i}_t)
\end{psmallmatrix}
$
for all $t\in [0,T]$ and $i=1,\ldots,m$,
and 
$I$ is the $(n+d)\t (n+d)$ identity matrix.

The regularization term $\frac{1}{m}I$ in \eqref{reg_ls}   guarantees the required matrix inverse  and  vanishes as $m \to \infty$. 
The estimator \eqref{reg_ls} can be equivalently  expressed as  an $\ell_2$-regularized least-squares estimator, as pointed out in \cite{duncan1992least} for the ergodic LQ-RL problem.

 
We summarize the continuous-time least-squares algorithm as follows.

\begin{algorithm}[H]
  \caption{{Continuous-time least-squares  algorithm}}
  \label{alg:conts_ls}
\begin{algorithmic}[1]
  \STATE \textbf{Input}: 
  Choose an initial estimation  
 $ \theta_0$ of $\theta^\star$
 and numbers of learning episodes $\{m_\ell\}_{\ell\in \sN\cup\{0\}}$. 

 \FOR {$\ell=0, 1, \cdots$}
  \STATE 
  Obtain the   feedback control $ \psi^{ \theta_\ell}$
  as \eqref{K2nonoise} with $\theta=\theta_\ell$.
  \STATE  Execute the feedback control $\psi^{ \theta_\ell}$ for $m_\ell$ independent episodes, 
   and collect the trajectory data
   $(X^{\psi^{ \theta_\ell},i}_t, \psi^{\theta_\ell}(t,X^{\psi^{ \theta_\ell},i}_t))_{t\in [0,T]}$, $i=1,\ldots,m_\ell$.

    \STATE Obtain an updated estimation ${\theta}_{\ell+1}$ by
    using  \eqref{reg_ls} and the $m_\ell$ trajectories collected above.
\ENDFOR 
\end{algorithmic}
\end{algorithm}

{\color{black}
Note that Algorithm \ref{alg:conts_ls} operates in cycles, with $m_\ell$ the number of episodes in the $\ell$-th cycle.
Hence,
the regret of learning up to $M$ episodes (cf.~\eqref{reg})
can be upper bounded by the accumulated 
regret at the end of the $L$-th cycle, 
where $L$ is the smallest integer such that 
$\sum_{\ell=0}^L m_\ell\ge M$.

}

In this section, we analyze the regret of Algorithm \ref{alg:conts_ls} based on the following assumptions of the learning problem \eqref{LQ}-\eqref{dynamics}.

\begin{Assumption}
\phantomsection
\label{assum:ls}

\begin{enumerate}[(1)]
\item 
\label{assum:pd}
$T\in (0,\infty)$,  $n,d\in \sN$,
$x_0\in \sR^n$,
$A^\star\in \sR^{n\t n}$,
$B^\star\in \sR^{n\t d}$,
$Q\in \sS^n_0$ and $R\in \sS^d_+$.
\item 
{\color{black}
\label{assum:non_degenerate}
  $\{v\in \sR^d\mid (K^{\theta^\star}_t)^\top v=0,
\;
\fa t\in [0,T]\}=\{0\}$,
with $K^{\theta^\star}$
defined in \eqref{K1}.
}
\end{enumerate}

\end{Assumption}

Before discussing  the regret  of Algorithm \ref{alg:conts_ls},
 we make the following remark
 of (H.\ref{assum:ls}).

\begin{Remark}
[\textbf{Self-exploration of finite-time horizon RL problems}]
\label{rmk:identificable}
 (H.\ref{assum:ls}\ref{assum:pd}) is the standard assumption 
 for finite-time horizon LQ-RL problems
(see e.g., \cite{hambly2020policy}),
{\color{black}
except that H.\ref{assum:ls}\ref{assum:pd} 
allows  $Q$ to be 
positive semidefinite,
which is important for costs depending on partial states.
}
(H.\ref{assum:ls}\ref{assum:non_degenerate}) corresponds to the identifiability of the true parameter  $\theta^\star$
by executing   the optimal policy  $K^{\theta^\star}$.
In fact, 
as shown
in  Proposition \ref{prop:non_degenerate},
under  (H.\ref{assum:ls}\ref{assum:pd}), 
  (H.\ref{assum:ls}\ref{assum:non_degenerate}) 
  is equivalent to 
 the following statement:
\begin{enumerate}
\item[(2')] \label{item:linear_span}
if $u\in \sR^n$ and $v\in \sR^d$ satisfy
$u^\top X^{\theta^\star}+v^\top U^{\theta^\star}=0$
 for
 $\d \sP\otimes \d t$-almost everywhere
 in $\Om\t [0,T]$, then $u=0$ and $v=0$,
 where 
 $X^{\theta^\star}$ and  $U^{\theta^\star}$ are
the optimal state and control processes of \eqref{LQ}-\eqref{dynamics}
defined by
\eqref{eq:state_psi_theta}
and 
\eqref{K1}, respectively,

\end{enumerate}

 Item 
 (2')
 indicates an important self-exploration  property of   
  finite-time horizon
 continuous-time  RL problems.
In particular, 
the time-dependent optimal feedback matrix 
$K^{\theta^\star}$
and the non-degenerate  noises
 guarantee the non-degeneracy of the space spanned by 
 $X^{\theta^\star}$ and 
  $U^{\theta^\star}$, 
enabling learning the parameters sufficiently well.
This self-exploration property is critical for our design  of exploration-free  learning algorithms for \eqref{LQ}-\eqref{dynamics}
with a logarithmic regret (see Theorems \ref{thm:regret_continuous} and \ref{thm:regret_discrete}).

{\color{black}
One can easily show that  (H.\ref{assum:ls}\ref{assum:non_degenerate}) 
holds if the optimal policy $(K^{\theta^\star})_{t\in [0,T]}$ is nondegenerate, i.e.,
$\sup_{t\in [0,T]}\lambda_{\min}\left((K^{\theta^\star}_t) (K^{\theta^\star}_t)^\top\right)>0$.
Similar nondegeneracy condition has been imposed in  \cite{cassel2020logarithmic} for discrete-time ergodic LQ-RL problems. In particular, by assuming that the optimal stationary policy satisfies  $\lambda_{\min}(K^\star(K^\star)^\top)>0$ (along with other controllablity conditions), they propose learning algorithms with a logarithmic regret,
under the assumption that only
the control coefficient
$B^\star$ is unknown. 
In contrast, 
we allow both the state coefficient $A^\star$ and the control coefficient $B^\star$ to be unknown. 

}


\end{Remark}

{\color{black}
Moreover, the following proposition provides sufficient conditions of  (H.\ref{assum:ls}\ref{assum:non_degenerate}),
which are special cases of 
Proposition
\ref{prop:non_degenerate_sufficient condition}.

\begin{Proposition}
\label{prop:non_degenerate_statement}
Let
$n,d\in \sN$, 
$Q\in \sS_0^n$ and $R\in \sS_+^d$.
\begin{enumerate}[(1)]
\item\label{item:sufficient_arbitraryT_statement}
 If
$(B^\star)^\top Q B^\star\in \sS^d_+$,
then (H.\ref{assum:ls}\ref{assum:non_degenerate}) holds for all  $T>0$.

\item
\label{item:sufficient_largeT_statement}
Assume that  the algebraic Riccati equation 
$
(A^\star)^\top P+PA^\star -P (B^\star R^{-1}(B^\star)^\top )P+Q=0
$
admits a 
unique  maximal  solution
$P^\star_\infty\in \sS^n_+$.
Let $K^\star_\infty=-R^{-1}(B^\star)^\top P^\star_\infty$, and
for each $T>0$, let $P^{\star,(T)}\in C([0,T];\sS_0^n)$ be defined in \eqref{riccati1}.
Assume that $\lim_{T\to \infty} P^{\star,(T)}_0=P^\star_\infty$ and $K^\star_\infty(K^\star_\infty)^\top\in \sS_+^d $.
Then there exists $T_0>0$, such that 
 (H.\ref{assum:ls}\ref{assum:non_degenerate}) holds 
for all $T\ge T_0$.
\end{enumerate}

\end{Proposition}

Proposition \ref{prop:non_degenerate_statement}
provides two sets of conditions for 
(H.\ref{assum:ls}\ref{assum:non_degenerate})
under two different scenarios: 
Item 
\ref{item:sufficient_arbitraryT_statement}
applies to an arbitrary finite $T>0$, 
and 
Item \ref{item:sufficient_largeT_statement}
only applies to 
sufficiently large $T$. Item \ref{item:sufficient_largeT_statement}
 assumes the asymptotic behavior of solutions to Riccati differential equations,
 which  can be ensured by 
 the stabilizability of 
 the pair
 $(A^\star,B^\star)$ and 
 detectability of  the pair
 $(A^\star, Q^{1/2})$
 (see \cite[Theorems 10.9 and 10.10]{bitmead1991riccati}).
 Note that our subsequent analysis is based on 
 (H.\ref{assum:ls}), 
 and does not require 
 stabilizability assumptions.
}

\begin{Remark}[\textbf{Stabilizability of  $(A^\star,B^\star)$
and dependence on $T$}]
Since  the LQ-RL problem \eqref{LQ}-\eqref{dynamics}
is
over the time horizon  $[0,T]$ with a fixed $T<\infty$, 
in general
one does not need  additional conditions on  
$(A^\star,B^\star)$ for  the well-definedness of \eqref{LQ}-\eqref{dynamics}.
If $T=\infty$, then some 
 stabilizability/controllability conditions of $(A^\star,B^\star)$ may be required 
for \eqref{LQ}-\eqref{dynamics} to ensure a well-defined solution
(see e.g., \cite{dean2019sample}).
{\color{black} Under these conditions, different algorithms have been shown to achieve sub-linear regret with respect to the number of decision steps (see e.g., \cite{certainty_ben,cohen2019learning}), and   even logarithmic regrets
provided that further identifiability assumptions are satisfied
(see e.g., \cite{kazem2018adaptive,faradonbeh2020input,cassel2020logarithmic,lale2020logarithmic});
see Section \ref{sec:discrete_time_rl} for more details.}
For $T<\infty$, 
the  regrets of learning algorithms 
for \eqref{LQ}-\eqref{dynamics}
 in general depend exponentially on the time horizon $T$
(e.g., the constants $C_0,C'$ in Theorem \ref{thm:regret_continuous}),
as
the moments of the optimal state process $X^{\theta^\star}$ and control process $U^{\theta^\star}$
may grow   exponentially with respect to $T$.
It would be interesting to quantify the precise dependence
of the regret bounds
 on $T$. This would entail
deriving precise \textit{a priori} bounds of solutions to \eqref{riccati2}
and 
estimating the norm 
$\|(\mathbb{E}[\int_0^T Z^{\theta^\star}_t (Z^{\theta^\star}_t)^\top\, \d t])^{-1}\|_2$ 
in terms of $(A^\star,B^\star,Q,T)$, and is  left for future research.

\end{Remark}
 
We are now ready to state the main result of this section, which shows that the regret of Algorithm \ref{alg:conts_ls} 
grows logarithmically   with respect to the number of episodes.

\begin{Theorem}
\label{thm:regret_continuous}
Suppose (H.\ref{assum:ls}) holds
and let
 $\theta_0=(A_0,B_0)^\top\in \sR^{(n+d)\t d}$
such that 
 (H.\ref{assum:ls}\ref{assum:non_degenerate}) holds with $\theta_0$.
 Then
there exists a constant $C_0>0$ such that 
 for all 
 $C\ge C_0$, and $\delta\in (0,\frac{3}{\pi^2})$,
 if one sets 
 $m_0=C(-\ln \delta)$ and $m_\ell=2^\ell m_0$ for all $\ell\in \sN$,
 then  with probability at least $1-\frac{\pi^2\delta}{3}$,
  the regret of Algorithm \ref{alg:conts_ls}
  given by 
\eqref{reg}
 satisfies
 $$
 R(M)\le C'\big((\ln M)(\ln\ln M)+(-\ln \delta) (\ln M) \big),
 \q \fa M\in \sN,
 $$ 
where  $C'$ is  a constant  independent of $M$ and $\delta$. 

\end{Theorem}

To simplify the presentation, we 
analyze the performance of  Algorithm \ref{alg:conts_ls}  
by assuming 
  the number of learning episodes $\{m_\ell\}_{\ell}$ 
  is doubled 
 between two successive updates of the estimation of $\theta^\star$.
Similar regret results can be established 
for Algorithm \ref{alg:conts_ls} with different choices of $\{m_\ell\}_{\ell}$.
{\color{black}
Under this specific choice of $\{m_\ell\}_{\ell}$, for any $M\in \sN$, Algorithm \ref{alg:conts_ls}  splits  $M$ episodes into $L=\lceil \log_2(\frac{M}{m_0}+1)\rceil-1$ 
cycles, 
where the $\ell$-th cycle, 
 $\ell=0,1,\ldots,L-1$,
contains
 $m_\ell$ episodes,
 and the remaining 
 $M-\sum_{\ell=0}^{L-1}m_\ell$
 episodes are in the last cycle.
 
 }

\paragraph{Sketched proof of Theorem \ref{thm:regret_continuous}.}

We  outline the key steps of the proof of 
Theorem \ref{thm:regret_continuous}, and present the detailed arguments 
to Section \ref{sec:regret_continuous}.
By exploiting the  regularity and robustness of solutions to 
\eqref{riccati2},
we prove that 
the performance gap 
$J^{\theta^\star}(U^{\psi^{\theta}})-J^{\theta^\star}(U^{\theta^\star})$
is of the magnitude $\cO(|\theta-\theta^\star|^2)$, for all 
 \textit{a-priori} bounded $\theta$ 
(Proposition \ref{prop:performance_gap}).
We then establish 
a uniform  sub-exponential property for 
 the
 (deterministic and stochastic)
  integrals in  \eqref{rewrTheta},
which 
along with 
(H.\ref{assum:ls}\ref{assum:non_degenerate}) 
and Bernstein's inequality
leads to the following estimate of the parameter estimation error: for all $\delta\in (0,1/2)$, 
all sufficiently large $m\in \sN$,
and all $\theta$ sufficiently close to $\theta^\star$,
\bb\label{eq:parameter_error_conts}
|\hat{\theta}-\theta^\star|\leq \cO\Big(\sqrt{\frac{-\ln\delta}{m}}+\frac{-\ln\delta }{m}+\frac{(-\ln\delta)^{2}}{m^2}\Big),
\q 
\textnormal{with probability $1-2\delta$,}
\ee
where $\hat{\theta}$ is generated by 
\eqref{reg_ls} with $\psi^\theta$
(Proposition \ref{theta_conc}). 
{\color{black} 
Then for each $\delta>0$,
applying \eqref{eq:parameter_error_conts} with $\delta_\ell=\delta/(\ell+1)^2$ for all  $\ell\in \sN\cup\{0\}$
shows that
with probability $1-2\sum_{\ell=0}^\infty\delta_\ell=1-\tfrac{\pi^2\delta}{3}$,
 \begin{equation}
 \label{eq:parameter_error_conts_sketched}
|\hat{\theta}_{\ell+1}-\theta^\star|^2\lesssim 
\frac{-\ln\delta_\ell}{m_\ell}
+\frac{(-\ln\delta_\ell)^2 }{m^2_\ell}
+\frac{(-\ln\delta_\ell)^{4}}{m_\ell^4},
\quad \forall \ell\in \sN,
\end{equation}
where $\lesssim$ means the inequality holds 
with a multiplicative constant independent of $\delta$ and $\ell$.
By the quadratic performance gap
and the choice of $\{m_\ell\}_\ell$, the regret of Algorithm \ref{alg:conts_ls}
up to the $M$-th episode can be bounded by
the  regret at the end of $L$-th cycle
with $L=\lceil \log_2(\frac{M}{m_0}+1)\rceil-1 $:
 \begin{align}
 \label{eq:R_N_conts_sketched}
\begin{split}
R(M)&
\lesssim
\sum_{\ell=0}^{L } m_\ell |\theta_\ell-\theta^\star|^2
\lesssim
\sum_{\ell=0}^{L}  (-\ln\delta_\ell)
\left(
1
+\frac{-\ln\delta_\ell }{m_\ell}
+\frac{(-\ln\delta_\ell)^{3}}{m_\ell^3}
\right).
\end{split}
\end{align}
Observe that 
the choices of
$\{\delta_\ell\}_\ell$
and 
$\{m_\ell\}_\ell$
  ensure
that 
$\sup_{\delta\in
(0,\frac{3}{\pi^2}),\ell\in \sN}\frac{-\ln\delta_\ell}{m_\ell}<\infty$.
Hence,
the right-hand side of 
\eqref{eq:R_N_conts_sketched} is of the magnitude
$\cO\left(
\sum_{\ell=0}^{L}  (-\ln\delta_\ell)\right)
$, 
which along with the choices of $\delta_\ell$
and $L$ leads to the desired  regret  bound;
see the end of Section \ref{sec:regret_continuous}
for more details.
 }

\subsection{Discrete-time least-squares algorithm and its regret bound}
\label{sec:discrete_ls}
{\color{black} Note that 
Algorithm \ref{alg:conts_ls} in Section \ref{sec:conts_ls}
requires executing  feedback controls and observing  corresponding state trajectories
continuously.  
A common practice to 
solve continuous-time RL problems 
is by assuming that at each learning episode the dynamics only evolves  
in discrete time, and then estimate parameters according to discrete-time RL algorithms
(see e.g., \cite{munos1998reinforcement, munos2000study,munos2006policy,tallec2019making}).
As the true dynamics evolves continuously, 
it is necessary   to 
quantify the impact of reaction stepsize on the algorithm performance.

In this section, we analyze the performance of  the above procedure for solving \eqref{LQ}-\eqref{dynamics}. 
We  adapt regularized least-squares algorithms for discrete-time LQ problems to the present setting,
and establish their regret bounds in terms of the discretization stepsize. 
Our analysis  shows that
 a proper scaling of 
the regularization term 
in the least-squares estimation 
in terms of   stepsize
is critical
for
a  robust performance with respect to different timescales.
}



More precisely,
for a given cycle (i.e., the index $\ell$ in Algorithm \ref{alg:conts_ls}),
let $\theta =(A,B)^\top\in \sR^{(n+d)\t n} $ be the current estimate of $\theta^\star $
in \eqref{dynamics},
and  let
$\{t_i\}_{i=0}^{N}$, $N\in \sN$,
be a uniform partition of $[0,T]$ 
with stepsize 
$\tau=T/N$.
  We then assume that
  \eqref{LQ}-\eqref{dynamics} is  piecewise constant between any two grid points $\{t_i\}_{i=0}^{N}$,
  choose actions and 
  make observations every $\tau$,
 and update the estimated parameter 
  based on these observations.
To this end, we  
consider  the following discrete-time LQ control problem with parameter $\theta$:
\begin{align}\label{disc_LQ}
\begin{split}
\inf_{U \in \cH^2_N (\sR^d)} J_{{N}}(U),
\q
\textnormal{with 
$ J_{{N}}(U) = \mathbb{E}\left[\sum_{i=0}^{N-1} \left((X^{U,\tau}_{t_i})^\top  QX^{U,\tau}_{t_i}+U_{t_i}^\top  RU_{t_i}\right)\tau
\right]$
},
\end{split}
\end{align}
where 
$\cH^2_N (\sR^d)=\{U\in \cH^2(\sR^d)\mid U_t=U_{t_i}, t\in [t_i,t_{i+1}), i=0,\ldots,N-1\}$,
and 
$(X^{U,\tau}_{t_i})_{i=0}^{N-1}$ are defined by
\bb\label{disc_dynamics}
  X^{U,\tau}_{t_{i+1}}-X^{U,\tau}_{t_{i}}=  (AX^{U,\tau}_{t_i}+  B U_{t_i})\tau+ W_{t_{i+1}}-W_{t_i}, 
\q i=0,\ldots N-1;
\q 
X^{U,\tau}_0=x_0.
\ee
{\color{black}
Note that for simplicity,
our strategy is constructed
by assuming 
a discrete-time dynamics
arising from  
an Euler discretization of 
\eqref{dynamics}
(with the estimated parameter $\theta$);
similar analysis can be performed 
with a high-order approximation of 
\eqref{LQ}-\eqref{dynamics}.
}

It is well-known that  (see e.g., \cite{bitmead1991riccati}),
 the optimal  control of \eqref{disc_LQ}-\eqref{disc_dynamics}
 is given by 
the following feedback form:
\bb\label{eq:feedback_theta_tau}
U_t=\psi^{ \theta,\tau}(t,X^{U,\tau}_t), 
\q 
\textnormal{
with 
$\psi^{ \theta,\tau}(t,x)= K^{\theta,\tau}_t x,
\q
\fa (t,x)\in [0,T)\t \sR^n,
$
}
\ee
 where
 $K^{\theta,\tau}:[0,T)\to \sR^{d\t n}$ 
  is the piecewise constant function (with stepsize $\tau=T/N$)
 defined by
\begin{align}
 P^{\theta,\tau}_{t_{i-1}} 
 &=  \tau Q+(I+\tau A)^\top  P^{\theta,\tau}_{t_{i}} (I+\tau A) - (I+\tau A) ^\top P^{\theta,\tau}_{t_{i}}\tau B( R+\tau B^\top   P^{\theta,\tau}_{t_{i}} B)^{-1}   B^\top   P^{\theta,\tau}_{t_{i}} (I+\tau A),
 \nonumber
 \\
 &\q\q \fa i=0,\dots,N-1;
 \q   P^{\theta,\tau}_{T}=0,
 \nonumber
 \\
 K^{\theta,\tau}_t&=-( R+\tau B^\top   P^{\theta,\tau}_{t_{i+1}} B)^{-1}  B^\top   P^{\theta,\tau}_{t_{i+1}}(I+\tau A),
\q t\in [t_i,t_{i+1}),\,  i=0,\ldots,N-1.
\label{eq:K_pi}
\end{align}

We then implement the  piecewise constant strategy $\psi^{ \theta,\tau}$
defined in \eqref{eq:feedback_theta_tau}
on the original system 
\eqref{dynamics}
for $m$ episodes,
and  update the  estimated  parameter $\theta$
by observing 
\eqref{dynamics}
with stepsize $\tau=T/N$. 
More precisely, 
let 
   $X^{\psi^{\theta, \tau}}\in \cS^2(\sR^n)$
be the state process associated with  $\psi^{ \theta,\tau}$: 
\begin{equation}
\label{eq:sde_theta_pi}
\d X_t=(A^\star X_t+B^\star K^{\theta,\tau}_{t_i} X_t)\,\d t+\d W_t, 
\q t\in [t_i,t_{i+1}],\, i=0,\ldots,N-1; 
\q 
X_0=x_0,
\end{equation}
and
 $(X^{\psi^{\theta, \tau},j}_t)_{t\in [0,T]}$, $j=1,\ldots,m$, $m\in \sN$,
 be $m$ independent   trajectories of   $X^{\psi^{\theta, \tau}}\in \cS^2(\sR^n)$,
 we  update the parameter $\theta$ according to the following discrete-time least-squares estimator:
 \begin{equation}\label{reg_ls_discrete_1}
\theta\leftarrow 
\argmin_{\theta\in \sR^{(n+d)\t n}} \sum_{j=1}^m\sum_{i=0}^{N-1}\|X^{\psi^{\theta, \tau},j}_{t_{i+1}}-X^{\psi^{\theta, \tau},j}_{t_{i}}-\tau \theta^\top Z^{\psi^{\theta, \tau},j}_{t_{i}}\|_2^2+\tau \tr(\theta^\top \theta),
 \end{equation}
  with $Z^{\psi^{\theta, \tau},j}_{t_{i}}\coloneqq\begin{psmallmatrix} X^{\psi^{\theta, \tau},j}_{t_{i}}\\K^{\theta,\tau}_{t_i}  X^{\psi^{\theta, \tau},j}_{t_{i}} \end{psmallmatrix}$
  for all $i,j$.
  {\color{black}
  The  update 
  \eqref{reg_ls_discrete_1}
  is consistent with the agent's assumption that the state evolves according to  \eqref{disc_dynamics}
  between two grid points.
  Setting the  derivative 
    (with respect to $\theta$)
of the right-hand side of \eqref{reg_ls_discrete_1}
 to zero leads to 
  $$
  -\sum_{j=1}^m\sum_{i=0}^{N-1}
  \tau Z^{\psi^{\theta, \tau},j}_{t_{i}}
  \left(
  \left(X^{\psi^{\theta, \tau},j}_{t_{i+1}}-X^{\psi^{\theta, \tau},j}_{t_{i}}\right)^\top
  -\tau (Z^{\psi^{\theta, \tau},j}_{t_{i}})^\top \theta
  \right)
  +\tau \theta=0.
  $$
  Dividing both sides by $\tau/m$
  and rearranging the terms 
  give the following equivalent expression of the discrete-time least squares estimator \eqref{reg_ls_discrete_1}:
  }
\begin{equation}\label{reg_ls_discrete}
{\theta} \longleftarrow 
\bigg(\frac{1}{m}\sum_{j=1}^m \sum_{i=0}^{N-1} Z^{\psi^{\theta, \tau},j}_{t_{i}}(Z^{\psi^{\theta, \tau},j}_{t_{i}})^\top \tau+\frac{1}{m} I\bigg)^{-1} 
\bigg(\frac{1}{m}\sum_{j=1}^m\sum_{i=0}^{N-1}
Z^{\psi^{\theta, \tau},j}_{t_{i}}\left(X^{\psi^{\theta, \tau},j}_{t_{i+1}}-X^{\psi^{\theta, \tau},j}_{t_{i}}\right)^\top \bigg).
 \end{equation}
 
 \begin{Remark}
 [\textbf{Scaling hyper-parameters with timescales}]
 \label{rmk:scaling}
  In 
  principle, 
 when applying
discrete-time RL algorithms
in a continuous environment,
it is critical 
 to adopt a proper scaling of the hyper-parameters 
for a  robust  performance
with respect to different timescales.
 Indeed,
scaling the regularization term  $\tr(\theta^\top \theta)$
in \eqref{reg_ls_discrete_1} 
with respect to  the stepsize $\tau$ 
is essential for  the 
robustness of  \eqref{reg_ls_discrete}
for all small stepsizes $\tau$.
{\color{black}
If one updates
  $\theta$ by minimizing the following $\ell_2$-regularized loss function
with a given hyper-parameter $\alpha<1$ such that
 \begin{equation}\label{reg_ls_discrete_1_alpha}
\argmin_{\theta\in \sR^{(n+d)\t n}} \sum_{j=1}^m\sum_{i=0}^{N-1}\|X^{\psi^{\theta, \tau},j}_{t_{i+1}}-X^{\psi^{\theta, \tau},j}_{t_{i}}-\tau \theta^\top Z^{\psi^{\theta, \tau},j}_{t_{i}}\|_2^2+ \tau^\alpha\tr(\theta^\top \theta),
 \end{equation}
 then
 the corresponding discrete-time  estimator is  given by
 \begin{equation*}
{\theta}^\tau \coloneqq
\bigg(\frac{1}{m}\sum_{j=1}^m \sum_{i=0}^{N-1} Z^{\psi^{\theta, \tau},j}_{t_{i}}(Z^{\psi^{\theta, \tau},j}_{t_{i}})^\top \tau+\frac{1}{\tau^{1-\alpha} m} I\bigg)^{-1} 
\bigg(\frac{1}{m}\sum_{j=1}^m\sum_{i=0}^{N-1}
Z^{\psi^{\theta, \tau},j}_{t_{i}}\left(X^{\psi^{\theta, \tau},j}_{t_{i+1}}-X^{\psi^{\theta, \tau},j}_{t_{i}}\right)^\top \bigg).
 \end{equation*}
 Observe that 
 for any given $m\in \sN$, the estimator $\theta^\tau $ degenerates  to zero as the stepsize $\tau$ tends to zero.
Hence, to ensure the viability of $\theta^\tau$ across   different timescales, 
the number of episodes $m$ has to increase
 appropriately when   $\tau$ tends to zero.
In contrast, 
by choosing $\alpha=1$ in
\eqref{reg_ls_discrete_1_alpha},
\eqref{reg_ls_discrete} admits a  
 continuous-time limit \eqref{reg_ls} as $\tau\to 0$,
 and leads to a learning algorithm in which the  episode
 numbers and the time stepsize can be chosen independently (see Theorem \ref{thm:regret_discrete}).

 }
\end{Remark}


We now summarize the discrete-time least-squares algorithm as follows.

\begin{algorithm}[H]
  \caption{Discrete-time least-squares  algorithm}
  \label{alg:ls_discrete}
\begin{algorithmic}[1]
 \STATE \textbf{Input}: 
  Choose an initial estimation  
 $ \theta_0$ of $\theta^\star$,
 numbers of learning episodes $\{m_\ell\}_{\ell\in \sN\cup\{0\}}$
and  numbers of intervention points  $\{N_\ell\}_{\ell\in \sN\cup\{0\}}$.

 \FOR {$\ell=0, 1, \cdots$}
  \STATE 
  Obtain the piecewise constant  control  $\psi^{ \theta_\ell,\tau_\ell}$  
as   \eqref{eq:feedback_theta_tau}
	with  $\tau=T/N_\ell$
	and $\theta=\theta_\ell$.
  \STATE  Execute the  control $\psi^{ \theta_\ell,\tau_\ell}$ for $m_\ell$ independent episodes, 
   and collect the  data
   $X^{\psi^{ \theta_\ell,\tau_\ell},j}_{t_i}$, $i=0,\ldots,N_\ell$, $j=1,\ldots, m_\ell$.

    \STATE Obtain an updated estimation ${\theta}_{\ell+1}$ by using
      \eqref{reg_ls_discrete} and the data
	 $(X^{\psi^{ \theta_\ell,\tau_\ell},j}_{t_i})_{i=0,\ldots,N_\ell,j=1,\ldots, m_\ell}$.
\ENDFOR 
\end{algorithmic}
\end{algorithm}

{\color{black}
Again, as the $\ell$-th cycle of Algorithm \ref{alg:ls_discrete} contains  $m_\ell$  episodes,
for each $M\in\sN$,
the regret of learning up to $M$ episodes (cf.~\eqref{reg})
can be upper bounded by the accumulated 
regret at the end of the $L$-th cycle, 
where $L$ is the smallest integer such that 
$\sum_{\ell=0}^L m_\ell\ge M$.
}
The following theorem is an analogue of Theorem \ref{thm:regret_continuous} for  Algorithm \ref{alg:ls_discrete}.

\begin{Theorem}
\label{thm:regret_discrete}
Suppose (H.\ref{assum:ls}) holds
and let
 $\theta_0=(A_0,B_0)^\top\in \sR^{(n+d)\t d}$
such that 
 (H.\ref{assum:ls}\ref{assum:non_degenerate}) holds with $\theta_0$.
 Then
there exists   $C_0>0$
and $n_0\in \sN$
 such that 
 for all 
 $C\ge C_0$, and 
 $\delta\in (0,\frac{3}{\pi^2})$,
 if one sets 
 $m_0=C(-\ln \delta)$, $m_\ell=2^\ell m_0$
 and $N_\ell\ge n_0$
  for all $\ell\in \sN\cup\{0\}$,
 then  
 with probability at least $1-\frac{\pi^2\delta}{3}$, 
\textcolor{black}
 {the regret of Algorithm \ref{alg:ls_discrete}
given by \eqref{reg}
}
satisfies  
 \begin{align}
 \label{eq:regret_discrete_statement}
 R(M)\le C'\left((\ln M)(\ln \ln M)+(-\ln \delta) (\ln M)+(-\ln \delta )\sum_{\ell=0}^{\ln M} 2^\ell N_\ell^{-2}
 \right),
 \q 
\fa M\in \sN,
\end{align}
where  $C'$ is  a constant  independent of $M$, $\delta$ and $(N_\ell)_{\ell\in \sN\cup\{0\}}$. 
\end{Theorem}

{\color{black}

\begin{Remark}
Theorem \ref{thm:regret_discrete} provides a general regret bound
of Algorithm \ref{alg:ls_discrete}
with any time discretization  steps $\{N_\ell\}_{\ell\geq 0}$, where $N_\ell$ is the number of intervention points in the  $\ell$-th cycle. 
Compared with Algorithm \ref{alg:conts_ls},
the regret of Algorithm \ref{alg:ls_discrete}
has an additional term 
$(-\ln\delta) \sum_{\ell=0}^{\ln M} 2^\ell N_\ell^{-2}$:
for each learning episode, 
one achieves a  sub-optimal loss by
adjusting her  policy 
in the discrete time
and also suffers from model misspecification error 
in parameter estimation 
from discrete-time observations. 
Specifically,
\begin{itemize}
    \item if 
    the time discretization  step is fixed for all cycles,
    i.e., 
    $N_\ell=T/\tau$ for all $\ell$, then the last  term
    of \eqref{eq:regret_discrete_statement}
    is of the magnitude:
    $$\cO\left((-\ln\delta)\sum_{\ell=0}^{\ln M} 2^\ell N_\ell^{-2}\right)=\cO\left((-\ln\delta)\tau^2\sum_{\ell=0}^{\ln M} 2^\ell\right)=\cO((-\ln \delta)\tau^2 M),$$
    and consequently 
    Algorithm \ref{alg:ls_discrete} achieves 
    a sub-optimal linear regret;
    \item if 
    the time discretization  step of the $\ell$-th cycle increases exponentially in terms of $\ell$,
    e.g., $N_\ell=\sqrt{2}^{\ell}N_0$ for   $\ell=1,\ldots, \ln M$, then 
    the last  term
    of \eqref{eq:regret_discrete_statement}
    is of the magnitude:
    $$\cO\left((-\ln\delta)\sum_{\ell=0}^{\ln M} 2^\ell N_\ell^{-2}\right)=\cO\left((-\ln\delta)\sum_{\ell=0}^{\ln M} N_0^{-2}\right)=\cO((-\ln\delta)\ln M),$$ which guarantees that the regret of Algorithm \ref{alg:ls_discrete} is still logarithmic in $M$.
\end{itemize}
\end{Remark}
}

\paragraph{Sketched proof of Theorem \ref{thm:regret_discrete}.}

We point out  the main differences
 between the proofs of Theorems 
\ref{thm:regret_continuous}-\ref{thm:regret_discrete},
and give 
the detailed proof of  Theorem \ref{thm:regret_discrete}   in Section \ref{sec:regret_discrete}.
Compared with Theorem \ref{thm:regret_continuous}, 
the essential challenges   in proving Theorem \ref{thm:regret_discrete}
are to quantify the precise dependence 
of the performance gap 
and the parameter estimation error
on the stepsize. 
To this end,
we first prove a first-order convergence of \eqref{eq:K_pi} to \eqref{K2nonoise}
as the stepsize tends to zero.
Then by 
 exploiting the  affine structure of   \eqref{eq:sde_theta_pi},
we establish
the following quadratic performance gap
for a piecewise constant policy $\psi^{\theta,\tau}$
(Proposition \ref{prop:performance_gap_discrete}):  
\bb\label{eq:gap_discrete_sketch}
J^{\theta^\star}(U^{\psi^{\theta,\tau}})-J^{\theta^\star}(U^{\theta^\star})
\le C(|\theta-\theta^\star|^2+\tau^2).
\ee

The analysis of the parameter estimation error 
 is somewhat  involved,
 as 
the state  trajectories 
are merely   $\a$-H\"{o}lder continuous in time with $\a<1/2$.
{\color{black}By leveraging the analytic expression of $X^{\psi^{\theta,\tau}}$, we first show the first-order convergence of $\hat{\theta}^{\tau}$ to $\theta^\star$ with
\begin{equation}\label{eq:ls_discrete_auxilary}
\hat{\theta}^{\tau}
\coloneqq
 \bigg(\sE\bigg[ \sum_{i=0}^{N-1} Z^{\psi^{\theta,\tau}}_{t_{i}}(Z^{\psi^{\theta,\tau}}_{t_{i}})^\top \tau\bigg]\bigg)^{-1} 
\bigg(\sE\bigg[\sum_{i=0}^{N-1}
Z^{\psi^{\theta,\tau}}_{t_{i}}
\left(X^{\psi^{\theta,\tau}}_{t_{i+1}}-X^{\psi^{\theta,\tau}}_{t_{i}}\right)^\top 
\bigg]
\bigg).
\end{equation}
We then 
prove that \eqref{reg_ls_discrete} enjoys a uniform sub-exponential tail bound for all $\theta$  close to $\theta^\star$ and  small $\tau$.
Comparing \eqref{reg_ls_discrete} with \eqref{eq:ls_discrete_auxilary}
and applying the above results allow for bounding the  estimation error of 
\eqref{reg_ls_discrete}
by \eqref{eq:parameter_error_conts}
with an additional $\cO(\tau)$ term
(Proposition \ref{theta_conc_discrete}).
}

\section{Proofs of Theorems \ref{thm:regret_continuous} and \ref{thm:regret_discrete}}

To simplify the notation,
for  any given 
$N,m\in \sN$
and 
 control 
$\psi:[0,T]\t \sR^n\to \sR^d$
that is affine in the spatial variable,
we introduce the following 
random variables
associated with continuous-time observations:
\begin{equation}\label{eq:rv_ls_psi_conts}
\begin{alignedat}{2}
V^{\psi}
&=\int_0^T 
Z^{\psi}_{t}(Z^{\psi}_{t})^\top\,\d t,
\qq
&
Y^{\psi}
 &
=\int_0^T 
Z^{\psi}_{t}(\d X^{\psi}_{t})^\top,
\\
V^{\psi,m}
&=\frac{1}{m}\sum_{j=1}^m\int_0^T 
Z^{\psi,j}_{t}(Z^{\psi,j}_{t})^\top\,\d t,
\qq
&
Y^{\psi,m}
 &=\frac{1}{m}\sum_{j=1}^m
\int_0^T 
Z^{\psi,j}_{t}(\d X^{\psi,j}_{t})^\top,
\end{alignedat}
\end{equation}
and the  random variables associated with 
discrete-time observations with  stepsize $\tau=T/N$:
\begin{equation}\label{eq:rv_ls_psi_discrete}
\begin{alignedat}{2}
V^{\psi,\tau}
&=\sum_{i=0}^{N-1}
Z^{\psi}_{t_i}(Z^{\psi}_{t_i})^\top\tau,
\qq
&
Y^{\psi,\tau}
& =
\sum_{i=0}^{N-1}
Z^{\psi}_{t_i}( X^{\psi}_{t_{i+1}}- X^{\psi}_{t_{i}})^\top,
\\
V^{\psi,\tau,m}
&=
\frac{1}{m}\sum_{j=1}^m
\sum_{i=0}^{N-1}
Z^{\psi,j}_{t_i}(Z^{\psi,j}_{t_i})^\top\tau,
\qq
&
Y^{\psi,\tau,m}
&=
\frac{1}{m}\sum_{j=1}^m
\sum_{i=0}^{N-1}
Z^{\psi,j}_{t_i}( X^{\psi,j}_{t_{i+1}}-X^{\psi,j}_{t_{i}})^\top,
\end{alignedat}
\end{equation}
where  
$X^{\psi}$ is the state process
associated with
the parameter $\theta^\star$ 
and the control $\psi$
(cf.~\eqref{dynamics-n}),
$Z^{\psi}_t=
\begin{psmallmatrix} X^{\psi}
\\
\psi(t,X^{\psi}_t)
 \end{psmallmatrix}
$
 for all $t\in [0,T]$,
and
 $(X^{\psi,j},Z^{\psi,j})_{j=1}^m$
are independent copies of 
 $(X^{\psi},Z^{\psi})$.

\subsection{Convergence and stability of Riccati equations and   feedback controls}

\begin{Lemma}\label{lemma:a_priori_riccati}
Suppose (H.\ref{assum:ls}\ref{assum:pd}) holds. 
Then for all $\theta=(A,B)^\top\in \sR^{(n+d)\t n}$,
the  Riccati equation
\begin{equation}
\label{riccati_theta}
\tfrac{\d}{\d t}P_t + A^\top P_t + P_tA  - P_t B R^{-1}B^\top P_t + Q=0, 
\q t \in [0,T];
\q 
P_T=0.
\end{equation}
 admits a unique solution $P^\theta\in C([0,T];\sR^{n\t n})$.
 Moreover,
the map $\sR^{(n+d)\t n}\ni\theta\mapsto P^\theta\in C^1([0,T];\sR^{n\t n})$ is continuously differentiable.
 
\end{Lemma}
\begin{proof}
It has been shown in
\cite[Corollary 2.10 on p.~297]{yong1999stochastic}
that 
under (H.\ref{assum:ls}\ref{assum:pd}), 
 for all $\theta=(A,B)^\top\in \sR^{(n+d)\t n}$,
\eqref{riccati_theta} admits a unique solution $P^\theta\in C([0,T];\sR^{n\t n})$
such that $P^\theta_t\in \sS^n_0$ for all $t\in [0,T]$.
It remains to 
to prove the continuous differentiability of $\theta\mapsto P^\theta$.

To this end, consider
the Banach spaces
$\sX= \sR^{(n+d)\t n}
\t C^1([0,T];\sR^{n\t n})$
and
$\sY=C([0,T];\sR^{n\t n})\t \sR^{n\t n}$,
and
the operator
$\Phi:\sX\to \sY$ defined by
 $$
 \sX
 \ni ( \theta, P)
 \mapsto
 \Phi(\theta, P)
 \coloneqq
 (F(\theta,P), P_T)
\in \sY,
$$
where 
$F(\theta, P)_t=
\tfrac{\d}{\d t}P_t + A^\top P_t + P_tA  - P_t B R^{-1}B^\top P_t + Q 
$ for all $t\in [0,T]$.
Observe that for all $\theta\in \sR^{(n+d)\t n}$,
$\Phi(P^\theta, \theta)=0$.
Moreover, one can easily show that
for any $(P,\theta)\in \sX$,
$\Phi$ is continuously Fr\'echet 
differentiable at 
$(P,\theta)$,
and
the partial derivative 
$\frac{\p}{\p P }\Phi (\theta, P): C^1([0,T];\sR^{n\t n})
\to \sY$
is a bounded linear operator
such that 
for all $ \tilde{P}\in C^1([0,T];\sR^{n\t n})$,
$$
\frac{\p}{\p P }\Phi (\theta, P)(\tilde{P})
=
\begin{pmatrix}
\left(
\tfrac{\d}{\d t}\tilde{P}_t + A^\top \tilde{P}_t + \tilde{P}_tA  - \tilde{P}_t B R^{-1}B^\top P_t 
- P_t B R^{-1}B^\top \tilde{P}_t
\right)_{t\in [0,T]}
\\
\tilde{P}_T
\end{pmatrix}
\in \sY.
$$
 Classical well-posedness results of linear differential equations 
 and 
the boundedness of $P$ imply that 
$\frac{\p}{\p P }\Phi (P,\theta): C^1([0,T];\sR^{n\t n})
\to \sY$ has a bounded inverse (and hence a bijection).
Thus, applying
 the implicit function theorem 
 (see \cite[Theorem 7.13-1]{ciarlet2013linear})
 to 
 $\Phi$ proves  that 
 $\sR^{(n+d)\t n}\ni \theta\mapsto P^\theta
 \in 
  C^1([0,T];\sR^{n\t n})$ is continuously differentiable.
%
%
\end{proof}

The following lemma establishes the stability of   the Riccati difference operator,
which is crucial for the subsequent convergence analysis.

\begin{Lemma}\label{lemma:a_priori_riccati_discrete}
Suppose (H.\ref{assum:ls}\ref{assum:pd}) holds. 
For each $\theta=(A,B)^\top\in \sR^{(n+d)\t n}$
and 
$N\in \sN$, 
let $\tau=T/N$ and 
 the function $\Gamma^\theta_\tau:\sS^n_0\to \sS^n_0$ 
such that   for all $P\in \sS^n_0$,
\begin{equation}\label{eq:one_step_riccati}
\Gamma^\theta_\tau(P)
\coloneqq
\tau Q+(I+\tau A)^\top  P(I+\tau A) - (I+\tau A) ^\top P\tau B(  R+\tau B^\top   P  B)^{-1}    B^\top   P(I+\tau A).
\end{equation}
Then for all $P, P'\in \sS^n_0$,
\begin{enumerate}[(1)]
\item \label{item:a_prior_bdd}
$\|\Gamma^\theta_\tau(P)\|_2\le \tau\|Q\|_2+(1+\tau \|A\|_2)^2\|P\|_2$,
\item \label{item:stable}
$\|\Gamma^\theta_\tau(P)-\Gamma^\theta_\tau(P')\|_2
\le 
\big(1+\tau \|R^{-1}\|_2  \|B\|_2^2 \max\{\|P\|_2,\|P'\|_2\}\big)^2(1+\tau \|A\|_2)^2
\|P-P'\|_2.
$
\end{enumerate}
\end{Lemma}
\begin{proof}
  Item \ref{item:a_prior_bdd}
  follows directly from
  the definition of $\Gamma^\theta_\tau$ and 
   the identity that $\|\Gamma^\theta_\tau(P)\|_2=\sup\{x^\top \Gamma^\theta_\tau(P)x\mid x\in \sR^n, |x|=1\}$.
We now prove   
Item \ref{item:stable}.
Let 
$\delta P=P-P'$
and 
$\delta \Gamma(P)=\Gamma^\theta_\tau(P)-\Gamma^\theta_\tau(P')$,
by
\cite[Lemma 10.1]{bitmead1991riccati},
\begin{align*}
\delta \Gamma(P)=F^\top\delta P F-F^\top\delta P \tau B(\tau B^\top P\tau B+\tau R)^{-1}\tau B^\top\delta P F,
\end{align*}
with 
$F=(I-\tau B(\tau B^\top P' \tau B+\tau R)^{-1} \tau B^\top P')(I+\tau A)$. 
Thus  
 for all $x\in \sR^n$,
$x^\top\delta \Gamma(P)x\le \|\delta P\|_2 \|F\|^2_2|x|^2$,
which along with  
$\|(\tau B^\top P'  B+ R)^{-1}\|_2\le \|R^{-1}\|_2 $
implies 
$$x^\top\delta \Gamma(P)x\le 
\|\delta P\|_2(1+\tau \|R^{-1}\|_2  \|B\|_2^2 \|P'\|_2)^2(1+\tau \|A\|_2)^2 
|x|^2,
\q x\in \sR^n.
$$
Hence,  interchanging the roles of $P$ and $P'$ in the above inequality and taking the supremum over $x\in \sR^n$ 
lead  to the desired estimate.
\end{proof}

The following proposition establishes the first-order convergence of the Riccati difference equation
and the associated feedback controls,
as the stepsize tends to zero.

\begin{Proposition}\label{prop:riccati_conv_rate}
Suppose (H.\ref{assum:ls}\ref{assum:pd}) holds and let $\Theta$ be a bounded subset of 
$\sR^{(n+d)\t n}$.
For each $\theta=(A,B)^\top\in\Theta$ and $N\in \sN$, 
let 
$(P^{\theta,\tau}_i)_{i=0}^N$ 
such that 
$P^{\theta,\tau}_N=0$ and 
$P^{\theta,\tau}_i=\Gamma^\theta_\tau(P^{\theta,\tau}_{i+1})$ for all $i=0,\ldots,N-1$,
with $\Gamma^\theta_\tau$  defined in \eqref{eq:one_step_riccati}
with $\tau=T/N$.
Then there exists a constant $C\ge 0$
such that
for all $\theta \in\Theta, N\in \sN$,
$$\sup_{i=0,\ldots, N-1}\sup_{t\in [i\tau,(i+1)\tau )}
\big(
\|P^\theta_t-P^{\theta,\tau}_i\|_2
+\|K^\theta_t-K^{\theta,\tau}_i\|_2
\big)
\le C\tau,
$$
where 
$P^\theta\in C^1([0,T];\sR^{n\t n})$ satisfies 
\eqref{riccati_theta},
$K^\theta_t=-R^{-1}B^\top P^{\theta}_t$ for all $t\in [0,T]$
and 
$K^{\theta,\tau}_i=
-( R+\tau B^\top   P^{\theta,\tau}_{{i+1}} B)^{-1}  B^\top   P^{\theta,\tau}_{{i+1}}(I+\tau A)$
for all $i=0,\ldots, N-1$.
\end{Proposition}
\begin{proof}
Throughout this proof, we shall fix $\theta\in \Theta, N\in \sN$, 
let  $t_i=i\tau$ for all $i=0,\ldots, N$,
and 
denote by
 $C$ a generic constant independent of $N$ and $\theta$.
By the continuity of the map $\theta\mapsto P^\theta$ 
(Lemma \ref{lemma:a_priori_riccati})
and the  boundedness of $\Theta$,  there exists a constant $C$ such that 
$\|P^\theta\|_{ C^1([0,T];\sR^{n\t n})}\le C$ for all $\theta\in \Theta$,
which implies 
$\|P^\theta_t-P^{\theta}_s\|_2\le C|t-s|$
 for all $t,s\in [0,T]$.
Consequently, 
it suffices to prove 
$\|P^\theta_{t_i}-P^{\theta,\tau}_i\|_2+\|K^\theta_{t_{i+1}}-K^{\theta,\tau}_i\|_2
\le C\tau
$
for all $i=0,\ldots, N-1$.

We start by making two important observations.
By Lemma 
\ref{lemma:a_priori_riccati_discrete} Item \ref{item:a_prior_bdd},
$\|P^{\theta,\tau}_i\|_2\le \tau C+(1+C\tau)\|P^{\theta,\tau}_{i+1}\|_2$
for all  $i=0,\ldots, N-1$,
which along with Gronwall's inequality gives 
$ \|P^{\theta,\tau}_i\|_2\le C$
for all $i=0,\ldots, N$.
Moreover, by  \eqref{eq:one_step_riccati},
for all $P\in \sS^n_0$,
\begin{align*}
\Gamma^\theta_\tau(P)
&=
\tau Q+
P+\tau (A^\top P+PA)+\tau^2 A^\top P A
\\
&\q -\tau \Big(
P  B (  R+\tau B^\top   P   B)^{-1}  B^\top P
+ \tau (A^\top H+HA^\top)+\tau^2 A^\top HA
\Big),
\end{align*}
with $H\coloneqq P  B (  R+\tau B^\top   P   B)^{-1}  B^\top P$.
Hence
for any given  $i=0,\ldots, N-1$,
we see  from 
\eqref{riccati_theta}
that 
\begin{align*}
&P^\theta_{t_i}-\Gamma^\theta_\tau(P^\theta_{t_{i+1}})
\\
&=
\int_{t_{i}}^{t_{i+1}}\big(A^\top (P^\theta_t -P^\theta_{t_{i+1}})+ (P^\theta_t-P^\theta_{t_{i+1}})A  \big)\,\d t
-\int_{t_{i}}^{t_{i+1}}  (P^\theta_t B R^{-1}B^\top P^\theta_t-P^\theta_ {t_{i+1}}B R^{-1}B^\top P^\theta_{t_{i+1}}) \,\d t
\\
&\q -
\int_{t_{i}}^{t_{i+1}}  (P^\theta_ {t_{i+1}}B R^{-1}B^\top P^\theta_{t_{i+1}}
-P^\theta_ {t_{i+1}}B (  R+\tau B^\top   P^\theta_{t_{i+1}}   B)^{-1} B^\top P^\theta_{t_{i+1}}
) \,\d t
\\
&\q +\tau^2 (
-A^\top P^\theta_ {t_{i+1}} A+
A^\top {H}^\theta_{i+1}+{H}^\theta_{i+1} A^\top+\tau A^\top {H}^\theta_{i+1} A),
\end{align*}
with 
${H}^\theta_{i+1}=P^\theta_ {t_{i+1}}B (  R+\tau B^\top   P^\theta_{t_{i+1}}   B)^{-1} B^\top P^\theta_{t_{i+1}}$.
Since $\|P^\theta\|_{ C^1([0,T];\sR^{n\t n})}\le C$ and $R\in \sS^d_+$, 
we have $\|P^\theta_{t_i}-\Gamma^\theta_\tau(P^\theta_{t_{i+1}})\|_2\le C\tau^2$ for all $i=0,\ldots, N-1$.

We are ready to  
show   $\max_{i=0,\ldots, N-1}
(\|P^\theta_{t_i}-P^{\theta,\tau}_i\|_2+\|K^\theta_{t_{i+1}}-K^{\theta,\tau}_i\|_2)
\le C\tau$.
For any given  $i=0,\ldots, N-1$, 
by Lemma \ref{lemma:a_priori_riccati_discrete} Item \ref{item:stable} 
 and 
 the uniform boundedness of $(P^\theta_{t_i})_{i=0}^N$ and $(P^{\theta,\tau}_i)_{i=0}^N$,
 \begin{align*}
 \|P^\theta_{t_i}-P^{\theta,\tau}_i\|_2
& \le 
\|P^\theta_{t_i}-\Gamma^\theta_\tau(P^\theta_{t_{i+1}})\|_2
 +\|\Gamma^\theta_\tau(P^\theta_{t_{i+1}})-
 \Gamma^\theta_\tau(P^{\theta,\tau}_{i+1})\|_2
 \\
& \le C\tau^2+
 \big(1+\tau  C \max\{\|P^\theta_{t_{i+1}}\|_2,\|P^{\theta,\tau}_{{i+1}}\|_2\}\big)^2(1+\tau)
\|P^\theta_{t_{i+1}}-P^{\theta,\tau}_{{i+1}}\|_2
\\
&
\le
C\tau^2+
 \big(1+\tau  C)\|P^\theta_{t_{i+1}}-P^{\theta,\tau}_{{i+1}}\|_2,
 \end{align*}
which 
 along with Gronwall's inequality 
 and $P^\theta_{T}=P^{\theta,\tau}_{N}=0$
 shows the desired convergence rate
 of $(P^{\theta,\tau}_{i})_{i=1}^N$. 
Furthermore, 
 for all $i=0,\ldots, N-1$,
 \begin{align*}
 \|K^\theta_{t_{i+1}}-K^{\theta,\tau}_i\|_2
 &\le \|
 (R^{-1}- ( R+\tau B^\top   P^{\theta,\tau}_{{i+1}} B)^{-1})B^\top P^{\theta}_{t_{i+1}}\|_2
 \\
&\q  +
\| ( R+\tau B^\top   P^{\theta,\tau}_{{i+1}} B)^{-1}  B^\top   ( P^{\theta}_{t_{i+1}}-P^{\theta,\tau}_{{i+1}}(I+\tau A))\|_2
\le 
C\tau,
 \end{align*}
from the facts that 
 $\|P^{\theta}_{t_i}\|_2\le C$,  $\|P^{\theta,\tau}_{i}\|_2\le C$
 and $\|P^\theta_{t_i}-P^{\theta,\tau}_i\|_2 \le C\tau$ for all $i$.
\end{proof}

\subsection{Concentration inequalities for  least-squares estimators}

In this section, we analyze the  concentration behavior of the least-squares estimators
\eqref{reg_ls}
and \eqref{reg_ls_discrete}.
We first recall the definition of  sub-exponential 
random variables
(see e.g., \cite{martin}).

\begin{Definition}
	A  random variable $X$ with mean $\mu=\mathbb{E}[X]$ is 
	$(\nu,b)$-sub-exponential
	for  $\nu,b\in [0,\infty)$
	 if $\mathbb{E}[e^{\lambda(X-\mu)}]\leq e^{\nu^2\lambda^2/2}$ for all $|\lambda|<1/b$.

\end{Definition}

Note that 
a $(\nu,0)$-sub-exponential random variable is usually called a sub-Gaussian random variable.
It is well-known that products of sub-Gaussian random variables are sub-exponential,
and   the class of sub-exponential random variables forms a vector space.
Moreover, 
sub-exponential random variables enjoy the following  concentration inequality
(also known as Bernstein's inequality; see e.g., \cite[Equation 2.18 p.~29]{martin}).
\begin{Lemma} \label{concentration}
Let $m\in \sN$,
$\nu,b\in [0,\infty)$ and 
$(X_i)_{i=1}^m$ be independent 
$(\nu,b)$-sub-exponential random variables
with $\mu=\sE[X_i]$ for all $i=1,\ldots, m$.
Then for all $\epsilon\ge 0$,
\[
\mathbb{P}\left(\left|\frac{1}{m}\sum_{i=1}^m X_i -\mu\right| \geq \epsilon\right)\leq 
2\exp\left(-\min\left\{\frac{m \epsilon^2}{2\nu^2},\frac{m \epsilon }{2b}\right\}\right).
\]
\end{Lemma}

The following lemma  shows  double iterated It\^{o} integrals
are   sub-exponential random variables.

\begin{Lemma}\label{uniform_sub_exp}
Let 
$L\ge 0$
and
$g,h:[0,T]\t [0,T]\to \sR^{n\t n}$ be   measurable functions
such that 
$|g(t,s)|\le L$ and 
$|h(t,s)|\le L$ for all   $t,s\in [0,T]$. Then 
there exist $\nu,b\in [0,\infty)$,
depending polynomially on $L, n, T$,
such that 
\begin{enumerate}[(1)]
\item
\label{item:subexp_g}
$\int_0^T \big(\int_0^t g(t,s)\,\d W_s\big)^\top\, \d W_t$,
\item
\label{item:subexp_gh}
$  \int_0^T \big(\int_0^t g(t,s)\,\d W_s\big)^\top
\big(\int_0^t h(t,s)\,\d W_s\big)
\, \d t$ 
\end{enumerate}
are  $(\nu,b)$-sub-exponential,
%
\end{Lemma}

\begin{proof}
We first prove Item \ref{item:subexp_g}
by assuming without loss of generality that 
$\|g(t,s)\|_2\le L$ for all $t,s\in [0,T]$,
and 
by defining
$V^q\coloneqq \int_0^T \left(\int_0^t q(t,s)\,\d W_s\right)^\top \d W_t$
for any bounded measurable function $q:[0,T]\t [0,T]\to \sR^{n\t n}$.
By  similar arguments as   \cite[Lemma 3.2]{subexp_uni}, 
we have   for all
$t\in [0,T]$ and
 $0\le \lambda<\frac{1}{2T}$, 
 $$\mathbb{E}[\exp(2\lambda V^{\frac{g}{L}})]\leq \mathbb{E}[\exp(2\lambda V^{I_n})]=\left(\frac{1}{\sqrt{1-2\lambda T }}\exp(-\lambda T)\right)^n.$$
As   $\frac{e^{-\lambda}}{\sqrt{1-2\lambda }} \leq e^{2\lambda^2}$ for all $|\lambda|\le 1/4$, 
we see 
 $\mathbb{E}[\exp(2\lambda V^{g/L})]\leq \exp(2n\lambda^2T^2)$
  for all  $0\le \lambda<\frac{1}{4T}$.
Consequently,
for all
$0\le \lambda<\frac{1}{2LT}$,
$$
\mathbb{E}[\exp(\lambda V^{g})] = 
\sE\left[\exp\left(2\frac{\lambda L}{2}V^{\frac{g}{L}}\right)\right]
\leq \exp\left(\frac{nL^2T^2\lambda^2}{2}\right).$$
Replacing $g$ by $-g$ shows the above estimate  holds for $| \lambda|<\frac{1}{2LT}$,
which implies  
the desired sub-exponential property of 
$V^g$.

For Item \ref{item:subexp_gh}, 
observe that 
for   each $t\in [0,T]$,
the  It\^{o} formula allows one to
express the product 
$\big(\int_0^t g(t,s)\,\d W_s\big)^\top
\big(\int_0^t h(t,s)\,\d W_s\big)$
as
a linear combination of double iterated It\^{o} integrals
and deterministic integrals.
Then 
the desired sub-exponential property follows from 
the stochastic Fubini theorem (see e.g., \cite{stochastic_fubini})
and Item \ref{item:subexp_g}.
\end{proof}

The following theorem establishes the concentration properties 
of the random variables involved in the least-squares estimators.

\begin{Theorem}\label{thm:concentration_ls}
Suppose (H.\ref{assum:ls}\ref{assum:pd}) holds and let $\Theta$ be a bounded subset of 
$\sR^{(n+d)\t n}$.
For each $\theta \in\Theta$ and $N\in \sN$, 
let 
$\psi^{\theta}$
be defined in \eqref{K2nonoise},
and 
$\psi^{\theta,\tau}$
be defined in \eqref{eq:K_pi}
with stepsize $\tau=T/N$.
Then there exist constants $C,\nu,b>0$ such that 
for all 
$\theta\in \Theta$,
$N,m\in \sN$
and $\epsilon>0$,
\begin{align*}
&\max\Big\{\sP(|V^{\psi^{\theta},m}-\sE[V^{\psi^{\theta}}]|\ge \epsilon),
\sP(|Y^{\psi^{\theta},m}-\sE[Y^{\psi^{\theta}}]|\ge \epsilon),
\\
&\q 
\sP(|V^{\psi^{\theta,\tau},\tau,m}-\sE[V^{\psi^{\theta,\tau},\tau}]|\ge \epsilon),
\sP(|Y^{\psi^{\theta,\tau},\tau,m}-\sE[Y^{\psi^{\theta,\tau},\tau}]|\ge \epsilon)\Big\}
\le C \exp\left(-\frac{1}{C} \min\left\{\frac{m \epsilon^2}{\nu^2},\frac{m \epsilon }{b}\right\}\right),
\end{align*}
where
$V^{\psi^{\theta}}$, $Y^{\psi^{\theta}}$,
$V^{\psi^{\theta},m}$, $Y^{\psi^{\theta},m}$
are defined in 
\eqref{eq:rv_ls_psi_conts},
and 
$V^{\psi^{\theta,\tau},\tau}$, $Y^{\psi^{\theta,\tau},\tau}$,
$V^{\psi^{\theta,\tau},\tau,m}$, $Y^{\psi^{\theta,\tau},\tau,m}$
are defined in
\eqref{eq:rv_ls_psi_discrete}.

\end{Theorem}
\begin{proof}
We first show 
there exist $\nu,b>0$ such that 
all entries of  
$V^{\psi^{\theta}}$,
$Y^{\psi^{\theta}}$
$Y^{\psi^{\theta,\tau},\tau}$,
$V^{\psi^{\theta,\tau},\tau}$
are $(\nu,b)$-sub-exponential
for all $\theta\in \Theta$ and $N\in \sN$.
By \eqref{eq:rv_ls_psi_conts}, we have
\begin{align*}
V^{\psi^{\theta}}
&=
\int_0^T 
\begin{pmatrix}
X^{\psi^{\theta}}_{t}
\\
K^\theta_t X^{\psi^{\theta}}_{t}
\end{pmatrix}
\begin{pmatrix}
(X^{\psi^{\theta}}_{t})^\top &
(K^\theta_t X^{\psi^{\theta}}_{t})^\top
\end{pmatrix}
\,\d t,
\q
Y^{\psi^{\theta}} 
=V^{\psi^{\theta}}
\theta^\star
+\int_0^T
\begin{pmatrix}
X^{\psi^{\theta}}_{t}
\\
K^\theta_t X^{\psi^{\theta}}_{t}
\end{pmatrix}
(\d W_{t})^\top.
\end{align*}
Moreover, 
applying the variation-of-constants formula
(see e.g., \cite[Theorem 3.1 p.~96]{mao2007stochastic})
to \eqref{lin_dyn} shows that 
$X^{\psi^{\theta}}_{t}=
\Phi^\theta_t \big(x_0+\int_0^t(\Phi^\theta_s)^{-1}\,\d W_s\big)
$ for all $t\in [0,T]$,
where 
$\Phi^\theta\in C([0,T]; \sR^{n\t n})$ is the fundamental  
solution of  
$\d \Phi^\theta_t=(A^\star+B^\star K^\theta_t)\Phi^\theta_t \,\d t$.
The continuity of $\sR^{(n+d)\t n}\ni \theta\mapsto K^\theta\in C([0,T];\sR^{d\t n})$ (cf.~Proposition \ref{lemma:a_priori_riccati})
and the boundedness of $\Theta$
implies that 
$K^\theta, \Phi^\theta, (\Phi^\theta)^{-1}$
are uniformly bounded 
 for all $\theta\in \Theta$.
Consequently,  from 
Lemma \ref{uniform_sub_exp}, 
 there exist $\nu,b>0$ such that 
all entries of 
$V^{\psi^{\theta}}$ and 
$Y^{\psi^{\theta}}$
are  $(\nu,b)$-sub-exponential. 

Similarly,  by  
 \eqref{eq:sde_theta_pi} 
 and \eqref{eq:rv_ls_psi_discrete}, 
\begin{align*}
V^{\psi^{\theta,\tau},\tau}
&=
\int_0^T 
\sum_{i=0}^{N-1}
\mathbf{1}_{[t_i,t_{i+1})}(t)
\begin{pmatrix}
X^{\psi^{\theta,\tau}}_{t_i}
\\
K^{\theta,\tau}_{t_i} X^{\psi^{\theta,\tau}}_{t_i}
\end{pmatrix}
\begin{pmatrix}
(X^{\psi^{\theta,\tau}}_{t_i})^\top &
(K^{\theta,\tau}_{t_i} X^{\psi^{\theta,\tau}}_{t_i})^\top
\end{pmatrix}
\,\d t,
\\
Y^{\psi^{\theta,\tau},\tau} 
&
=\int_0^T 
\sum_{i=0}^{N-1}
\mathbf{1}_{[t_i,t_{i+1})}(t)
\begin{pmatrix}
X^{\psi^{\theta,\tau}}_{t_i}
\\
K^{\theta,\tau}_{t_i} X^{\psi^{\theta,\tau}}_{t_i}
\end{pmatrix}
\begin{pmatrix}
(X^{\psi^{\theta,\tau}}_{t})^\top &
(K^{\theta,\tau}_t X^{\psi^{\theta,\tau}}_{t})^\top
\end{pmatrix}
(\theta^\star)^\top\,\d t
\\
&\q +\int_0^T
\sum_{i=0}^{N-1}
\mathbf{1}_{[t_i,t_{i+1})}(t)
\begin{pmatrix}
X^{\psi^{\theta,\tau}}_{t_i}
\\
K^{\theta,\tau}_{t_i} X^{\psi^{\theta,\tau}}_{t_i}
\end{pmatrix}
(\d W_{t})^\top,
\end{align*}
where 
$X^{\psi^{\theta,\tau}}_{t}=
\Phi^{\theta,\tau}_t \big(x_0+\int_0^t(\Phi^{\theta,\tau}_s)^{-1}\,\d W_s\big)
$ for all $t\in [0,T]$,
and 
$\Phi^{\theta,\tau}\in C([0,T]; \sR^{n\t n})$ is the fundamental  
solution of  
$\d \Phi^{\theta,\tau}_t=(A^\star+B^\star K^{\theta,\tau}_t)\Phi^{\theta,\tau}_t \,\d t$.
By  Proposition \ref{prop:riccati_conv_rate},
$K^{\theta,\tau}, \Phi^{\theta,\tau}, (\Phi^{\theta,\tau})^{-1}$ are uniformly bounded for all $\theta\in \Theta$ 
and $N\in \sN$,
which along with
Lemma \ref{uniform_sub_exp} 
leads to 
the desired sub-exponential properties of 
$Y^{\psi^{\theta,\tau},\tau}$ and
$V^{\psi^{\theta,\tau},\tau}$.

Finally, since  
$\sP(|\sum_{i=1}^\ell X_i|\ge \epsilon)
\le \sum_{i=1}^\ell
\sP(|X_i|\ge \epsilon/\ell)$
for all
$\ell\in \sN$ and 
 random variables $(X_i)_{i=1}^\ell$, 
 we can apply
Lemma \ref{concentration}
to each component of 
$V^{\psi^{\theta}}$,
$Y^{\psi^{\theta}}$
$Y^{\psi^{\theta,\tau},\tau}$ and
$V^{\psi^{\theta,\tau},\tau}$,
and 
 conclude the desired concentration inequality 
with a constant $C$ depending polynomially on $n,d$.
\end{proof}

\subsection{Regret analysis of continuous-time least-squares algorithm}
\label{sec:regret_continuous}

This section is devoted to the proof of 
 Theorem \ref{thm:regret_continuous}, which consists of three steps:
(1)
{\color{black}
We first 
quantify  
the performance gap 
between applying feedback controls for an incorrect model and that for the true model;
our proof exploits  the stability
of Riccati equations established  in  Lemma \ref{lemma:a_priori_riccati};
}
(2) 
We then 
 estimate the parameter estimation error 
in terms of the number of learning episodes
based on the sub-exponential tail behavior of the least-squares estimator \eqref{reg_ls};
(3) Finally, we estimate the regret    for the feedback controls $(\psi^{ \theta_\ell})_{\ell\in \sN}$ in Algorithm \ref{alg:conts_ls}, thus establishing Theorem \ref{thm:regret_continuous}.

\paragraph{Step 1: Analysis of the performance gap.}

We start by  establishing a quadratic expansion of the cost function
at any open-loop control.

\begin{Proposition}\label{prop:taylor_expansion}

Suppose (H.\ref{assum:ls}\ref{assum:pd}) holds.
Let 
 $\psi^{\theta^\star}$ be defined in
\eqref{K1},
$X^{{\theta^\star}}$
be  the state process
 associated with   
$\psi^{ \theta^\star}$ (cf.~\eqref{eq:state_psi_theta}),
and 
$U^{\theta^\star}\in \cH^2(\sR^d)$
be such that 
for all  $t\in [0,T]$,
$U^{\theta^\star}_t=\psi^{\theta^\star}(t, X^{{\theta^\star}}_t)$.
Then for all $U\in \cH^2(\sR^d)$,
\bb\label{eq:quadratic_gap}
J^{\theta^\star}(U)-J^{\theta^\star}(U^{\theta^\star})\leq 
\|Q\|_2\|X^{\theta^\star,U}-X^{{\theta^\star}}\|^2_{\cH^2(\sR^n)}+\|R\|_2
\|U-U^{\theta^\star}\|^2_{\cH^2(\sR^d)},
\ee
where  
$X^{\theta^\star,U}$ is the state process controlled by $U$ (cf.~\eqref{dynamics}),
and 
$J^{\theta^\star}:\cH^2(\sR^d)\to \sR$ is defined in \eqref{LQ}.

\end{Proposition}

\begin{proof}
For notational simplicity, 
for all 
 $U\in \cH^2(\sR^d)$
and $\epsilon>0$,  
we write 
 $U^\epsilon=U^{{\theta^\star}}+\epsilon (U-U^{{\theta^\star}})$,
 denote by 
$X^\epsilon=X^{\theta^\star,U^\epsilon}$
 the associated state process defined by \eqref{dynamics},
 and by $X^U=X^0=X^{\theta^\star,U}$.
The affineness of \eqref{dynamics} implies that 
$X^\epsilon=(1-\epsilon)X^{\theta^\star}+\epsilon X^U$ for all $\epsilon>0$. Hence,
for all $U\in \cH^2(\sR^d)$,
\begin{align*}
&\lim_{\epsilon\searrow 0}\frac{1}{\epsilon}\left(J^{\theta^\star}(U^\epsilon)-J^{\theta^\star}(U^{\theta^\star})\right)\\
&=\lim_{\epsilon\searrow 0}\frac{1}{\epsilon}
\mathbb{E}\bigg[\int_0^T
\bigg(
 \Big((1-\epsilon)X^{\theta^\star}_t+\epsilon X^U_t\Big)^{\top} 
Q
\Big((1-\epsilon)X^{\theta^\star}_t+\epsilon X^U_t\Big) 
-(X^{\theta^\star}_t)^{\top} Q X^{\theta^\star}_t 
 \\
&\q +
 \left((1-\epsilon)U^{{\theta^\star}}_t+\epsilon U_t\right)^{\top} R
 \left((1-\epsilon)U^{{\theta^\star}}_t+\epsilon U_t\right)
 -
 (U^{{\theta^\star}}_t)^{\top}  RU^{{\theta^\star}}_t
 \bigg)
 \,\d t
 \bigg]\\
&=\lim_{\epsilon\searrow 0} \epsilon \, \mathbb{E}\left[\int_0^T
\Big(
(X^U_t-X^{\theta^\star}_t)^{\top} Q(X^U_t-X^{\theta^\star}_t)+(U_t-U_t^{{\theta^\star}})^{\top} R(U_t-U_t^{{\theta^\star}})
\Big)\, \d t\right] \\
& \qquad +2 \mathbb{E}\left[\int_0^T \left((X^U_t-X^{\theta^\star}_t)^{\top} QX_t^{\theta^\star}+(U_t-U_t^{{\theta^\star}})^{\top} RU_t^{{\theta^\star}} \right)\,\d t\right]\\
&= 2 \mathbb{E}\left[\int_0^T \left((X^U_t-X^{\theta^\star}_t)^{\top} QX_t^{\theta^\star}+(U_t-U_t^{{\theta^\star}})^{\top} RU_t^{{\theta^\star}} \right)\,\d  t\right],
\end{align*}
which is based on the fact that  $X^U-X^{\theta^\star}\in {\cH^2(\sR^n)}$ and $U-U^\star\in {\cH^2(\sR^d)}$.
As
$U^\theta$ is the optimal control of $J^{\theta^\star}$,
 $ J^{\theta^\star}({U})\ge J^{\theta^\star}(U^{\theta^\star})$ for all ${U}\in \cH^2(\sR^d)$.
 Hence 
 for all $U\in \cH^2(\sR^d)$,
\bb\label{eq:one_side_limit}
\mathbb{E}\left[\int_0^T \left((X^U_t-X^{\theta^\star}_t)^{\top} QX_t^{\theta^\star}+(U_t-U_t^{{\theta^\star}})^{\top} RU_t^{{\theta^\star}} \right)\,\d  t\right]
=
\lim_{\epsilon\searrow 0 }\frac{1}{2\epsilon}(J^{\theta^\star}(U^\epsilon)-J^{\theta^\star}(U^{{\theta^\star}}))
\ge 0.
\ee
{\color{black}
We now prove that the above quantity is in fact zero for all $U\in \cH^2(\sR^d)$. 
To this end, let $U\in \cH^2(\sR^d)$ be a given (open-loop) control, and consider 
$\tilde{U}=U^{{\theta^\star}}-(U-U^{{\theta^\star}})$. 
Then by the  affineness of \eqref{dynamics},
$X^{\tilde{U}} -X^{{\theta^\star}}$ satisfies the following controlled dynamics:
\bb\label{eq:U_U^theta_dynamics}
\d X_t=(A^\star X_t- B^\star (U-U^{{\theta^\star}})_t)\,\d t, \q t\in [0,T];
\q X_0=0.
\ee
Moreover,
one can  verify 
by the  affineness of \eqref{dynamics} that
$-(X^{{U}} -X^{{\theta^\star}})$ also satisfies the  dynamics \eqref{eq:U_U^theta_dynamics},
which along with 
the uniqueness of solutions to 
\eqref{eq:U_U^theta_dynamics}
shows that $X^{\tilde{U}} -X^{{\theta^\star}}=-(X^{{U}} -X^{{\theta^\star}})$.
Therefore, applying \eqref{eq:one_side_limit} with $U=\tilde{U}$ implies that 
\begin{align*}
&0\le  \mathbb{E}\left[\int_0^T \left((X^{\tilde{U}}_t-X^{\theta^\star}_t)^{\top} QX_t^{\theta^\star}
 +(\tilde{U}_t-U_t^{{\theta^\star}})^{\top} RU_t^{{\theta^\star}} \right)\,\d  t\right]
 \\
 &=- \mathbb{E}\left[\int_0^T \left((X^{{U}}_t-X^{\theta^\star}_t)^{\top} QX_t^{\theta^\star}
 +({U}_t-U_t^{{\theta^\star}})^{\top} RU_t^{{\theta^\star}} \right)\,\d  t\right]
 \le 0.
\end{align*}
Hence for all $U\in \cH^2(\sR^d)$,
$$
 \mathbb{E}\left[\int_0^T \left((X^U_t-X^{\theta^\star}_t)^{\top} QX_t^{\theta^\star}+(U_t-U_t^{{\theta^\star}})^{\top} RU_t^{{\theta^\star}} \right)\,\d  t\right]=0,
 $$
}
which leads to the desired result  \eqref{eq:quadratic_gap}  due to the following identify:
\[
\begin{split}
 J^{\theta^\star}(U)-J^{\theta^\star}(U^{\theta^\star})
&
=\mathbb{E}\left[\int_0^T ((X^U_t)^{\top} QX^U_t
-(X^{\theta^\star}_t)^\top QX^{\theta^\star}_t
+U_t^{\top} RU_t-(U_t^{\theta^\star})^\top RU_t^{\theta^\star})\,\d t\right]\\
&=\mathbb{E}\left[\int_0^T
\left(
(X^U_t)^{\top} QX^U_t-(X^{\theta^\star}_t)^\top QX^{\theta^\star}_t+U_t^{\top} RU_t-(U_t^{\theta^\star})^\top RU_t^{\theta^\star}
\right)\,\d t\right]
\\
&\q 
-
2\mathbb{E}\left[\int_0^T
\left((X^U_t-X^{\theta^\star}_t)^{\top} QX_t^{\theta^\star}+(U_t-U_t^{{\theta^\star}})^{\top} RU_t^{{\theta^\star}} \right)\,\d  t
\right]\\
&=\mathbb{E}\left[\int_0^T
\left(
(X^U_t-X_t^{\theta^\star})^{\top} Q (X^U_t-X_t^{\theta^\star})  +(U_t-U_t^{\theta^\star})^{\top} R(U_t-U_t^{\theta^\star})
\,\d t
\right)
\right].
\end{split}
\]
\end{proof}

Armed with Proposition \ref{prop:taylor_expansion},
the following proposition quantifies the quadratic performance gap of a greedy policy $\psi^\theta$.
\begin{Proposition}\label{prop:performance_gap}
Suppose (H.\ref{assum:ls}\ref{assum:pd}) holds
and 
let  $\Theta$ be a bounded subset of 
$\sR^{(n+d)\t n}$.
For each $\theta \in\Theta$, 
let 
$\psi^{\theta}$
be defined in \eqref{K2nonoise},
let
$X^{\psi^{\theta}}$
be  the state process
 associated with    $\psi^{ \theta}$
 (cf.~\eqref{lin_dyn}),
let $\psi^{\theta^\star}$ be defined in
\eqref{K1},
and let  
$X^{{\theta^\star}}$
be  the state process
 associated with   
$\psi^{ \theta^\star}$ (cf.~\eqref{eq:state_psi_theta}).
Then there exists a constant $C$ such that 
$$
|J^{\theta^\star}(U^{\psi^{\theta}})-J^{\theta^\star}(U^{\theta^\star})|
\le C|\theta-\theta^\star|^2, \q \fa \theta\in \Theta,
$$
where $U^{\psi^{\theta}}_t=\psi^{\theta}(t,X^{\psi^{\theta}}_t)$ and 
$U^{\theta^\star}_t=\psi^{\theta^\star}(t, X^{{\theta^\star}}_t)$ for all $t\in [0,T]$,
and  $J^{\theta^\star}$ is defined in \eqref{LQ}.

\end{Proposition}

\begin{proof}
For all $\theta\in \Theta$,
applying  Proposition \ref{prop:taylor_expansion}
with $U=U^{\psi^\theta}$ gives  
\begin{align}\label{eq:quadratic_gap_feedback}
\begin{split}
&J^{\theta^\star}(U^{\psi^\theta})-J^{\theta^\star}(U^{\theta^\star})
\\
&
\leq  \|Q\|_2\|X^{\theta^\star, U^{\psi^\theta}}-X^{{\theta^\star}}\|^2_{\cH^2(\sR^n)}+\|R\|_2
\|U^{\psi^\theta}-U^{\theta^\star}\|^2_{\cH^2(\sR^d)},
\\
&\leq  \|Q\|_2\|X^{{\psi^\theta}}-X^{\psi^{\theta^\star}}\|^2_{\cH^2(\sR^n)}+\|R\|_2
\|\psi^{\theta}(\cdot, X^{\psi^{\theta}}_\cdot)-\psi^{\theta^\star}(\cdot, X^{\psi^{\theta^\star}}_\cdot)\|^2_{\cH^2(\sR^d)},
\end{split}
\end{align}
where 
the last inequality used the fact that $X^{\theta^\star, U^{\psi^\theta}}=X^{{\psi^\theta}}$
(see  \eqref{lin_dyn}),
and the definitions of 
 $U^{\psi^{\theta}}$ and 
$U^{\theta^\star}$.
It remains to  
prove 
$$
\|X^{{\psi^\theta}}-X^{\psi^{\theta^\star}}\|_{\cH^2(\sR^n)}+
\|\psi^{\theta}(\cdot, X^{\psi^{\theta}}_\cdot)-\psi^{\theta^\star}(\cdot, X^{\psi^{\theta^\star}}_\cdot)\|_{\cH^2(\sR^d)}
\le C|\theta-\theta^\star|,
$$ 
for a constant $C$ independent of $\theta$.
 
Observe that 
by  \eqref{K2nonoise},
for all $(t,x)\in [0,T]\t \sR^n$,
 $\psi^\theta(t,x)= K^\theta_t x$
 with $K^\theta_t= -R^{-1} B^\top  P^\theta_t$.
Now 
by Lemma \ref{lemma:a_priori_riccati}
and the boundedness of $\Theta$,
there exists a constant $C\ge 0$  such that 
$\|P^\theta\|_{C([0,T;\sR^{n\t n})}\le C$ and 
$\|P^{\theta}-P^{\theta^\star}\|_{C([0,T;\sR^{n\t n})}\le C|\theta-\theta^\star|$
for all $\theta\in \Theta\cup\{\theta^\star\}$,
which along with $K^\theta_t=-R^{-1}B^\top P^\theta_t$ implies that 
$\|K^\theta\|_{C([0,T;\sR^{d\t n})}\le C$ and 
$\|K^{\theta}-K^{\theta^\star}\|_{C([0,T;\sR^{d\t n})}\le C|\theta-\theta^\star|$.
Moreover,  observe  from \eqref{eq:state_psi_theta} and  \eqref{lin_dyn} that
$X^{{\theta^\star}}_0= X^{\psi^{\theta}}_0$ and for all $t\in [0,T]$,
$$
\d(X^{\psi^{\theta^\star}}- X^{\psi^{\theta}})_t
=
\Big((A^\star+B^\star K^{\theta^\star}_t)(X^{\psi^{\theta^\star}}- X^{\psi^{\theta}})_t
+B^\star (K^{\theta^\star}_t-K^\theta_t)  X^{\psi^{\theta}}_t
\Big)\,\d t,
$$
which combined with
the boundedness of $K^{\theta^\star}$
and
  Gronwall's inequality leads to 
 \begin{align*}
 \|X^{\psi^{\theta^\star}}- X^{\psi^{\theta}}\|_{\cH^2(\sR^n)}
& \le C
\|X^{\psi^{\theta^\star}}- X^{\psi^{\theta}}\|_{\cS^2(\sR^n)}
\\
&\le C 
\|(K^{\theta^\star}-K^\theta)  X^{\psi^{\theta}}\|_{\cH^2(\sR^{d})}
\le C 
\|K^{\theta^\star}-K^\theta\|_{C([0,T;\sR^{d\t n})} \| X^{\psi^{\theta}}\|_{\cH^2(\sR^{n})}
\\
&\le C|\theta-\theta^\star|,
\q \fa \theta\in \Theta,
\end{align*}
where
 the last inequality
follows from  $\|X^{\psi^{\theta}}\|_{\cH^2(\sR^{n})}\le C$, 
as 
$K^\theta$ is uniformly  bounded.
The above inequality further implies  
\begin{align*}
& \|\psi^{\theta}(\cdot, X^{\psi^{\theta}}_\cdot)-\psi^{\theta^\star}(\cdot, X^{\psi^{\theta^\star}}_\cdot)\|_{\cH^2(\sR^d)}
=\|K^{\theta}_\cdot X^{\psi^{\theta}}_\cdot-K^{\theta^\star}_\cdot  X^{\psi^{\theta^\star}}_\cdot\|_{\cH^2(\sR^d)}
\\
&\le 
\|( K^{\theta}_\cdot -K^{\theta^\star}_\cdot ) X^{\psi^{\theta}}_\cdot\|_{\cH^2(\sR^d)}
+
\|
K^{\theta^\star}_\cdot  (X^{\psi^{\theta}}_\cdot - X^{\psi^{\theta^\star}}_\cdot)\|_{\cH^2(\sR^d)}
\\
&\le   
\|K^{\theta^\star}-K^\theta\|_{C([0,T;\sR^{d\t n})} \| X^{\psi^{\theta}}\|_{\cH^2(\sR^{n})}
+\|
K^{\theta^\star}\|_{C([0,T;\sR^{d\t n})} \| X^{\psi^{\theta}} - X^{\psi^{\theta^\star}}\|_{\cH^2(\sR^n)}
\\
&\le C|\theta-\theta^\star|,
\q \fa \theta\in \Theta,
\end{align*}
which along with \eqref{eq:quadratic_gap_feedback} finishes the desired estimate.
\end{proof}

\paragraph{Step 2: Error bound for parameter estimation.}

\begin{Proposition}\label{theta_conc}

Suppose  (H.\ref{assum:ls}\ref{assum:pd}) holds
and  let  $\Theta\subset \sR^{(n+d)\t n}$
such that 
 there exists
  $C_1>0$ 
satisfying
$\|(\sE[V^{\psi^{{\theta}}}])^{-1}\|_2\le C_1$
and $|\theta|\le C_1$
for all $\theta\in \Theta$,
with $V^{\psi^\theta}$  defined in  \eqref{eq:rv_ls_psi_conts}.
Then there exist constants  $\bar{C}_1,\bar{C}_2\ge 0$,
such that 
for all
$\theta\in \Theta$ and
 $\delta\in (0,1/2)$, if 
$m\geq \bar{C}_1(-\ln\delta )$, 
 then   with probability at least $1-2\delta$, 
\bb\label{eq:parameter_estimation}
|\hat{\theta}-\theta^\star|\leq \bar{C}_2\bigg(\sqrt{\frac{-\ln\delta}{m}}+\frac{-\ln\delta }{m}+\frac{(-\ln\delta)^{2}}{m^2}\bigg),
\ee
where 
$\hat{\theta}$ denotes the  right-hand side of \eqref{reg_ls}
 with the control $\psi^\theta$.
\end{Proposition}

\begin{proof}
Let us fix  $\delta\in(0,1/2)$ and $\theta\in \Theta$.
By 
\eqref{rewrTheta} and
\eqref{reg_ls}, we obtain
\begin{equation*}
\begin{split}
\|\hat{\theta}-\theta^\star\|_2
& = 
\|(V^{\psi^{{\theta}},m}+\tfrac{1}{m}I)^{-1}Y^{\psi^{{\theta}},m}-(\sE[V^{\psi^{{\theta}}}])^{-1}\sE[Y^{\psi^{{\theta}}}]\|_2
\\
&\le 
\|(V^{\psi^{{\theta}},m}+\tfrac{1}{m}I)^{-1}-(\sE[V^{\psi^{{\theta}}}])^{-1}\|_2\|Y^{\psi^{{\theta}},m}\|_2
+\|(\sE[V^{\psi^{{\theta}}}])^{-1}\|_2\|Y^{\psi^{{\theta}},m}-\sE[Y^{\psi^{{\theta}}}]\|_2.
\end{split}
\end{equation*}
As 
$E^{-1}-F^{-1}=F^{-1}(F-E)E^{-1}$
for all nonsingular matrices $E$ and $F$,
we have 
\begin{equation}
\label{eq:est}
\begin{split}
&\|\hat{\theta}-\theta^\star\|_2
\\
&\le 
\|(V^{\psi^{{\theta}},m}+\tfrac{1}{m}I)^{-1}\|_2\|(\sE[V^{\psi^{{\theta}}}])^{-1}\|_2\|Y^{\psi^{{\theta}},m}\|_2
\|V^{\psi^{{\theta}},m}-\sE[V^{\psi^{{\theta}}}]+\tfrac{1}{m}I\|_2
\\
&\q +\|(\sE[V^{\psi^{{\theta}}}])^{-1}\|_2\|Y^{\psi^{{\theta}},m}-\sE[Y^{\psi^{{\theta}}}]\|_2
\\
&\le
C_1\big(
\|(V^{\psi^{{\theta}},m}+\tfrac{1}{m}I)^{-1}\|_2 \|Y^{\psi^{{\theta}},m}\|_2
\|V^{\psi^{{\theta}},m}-\sE[V^{\psi^{{\theta}}}]+\tfrac{1}{m}I\|_2
+\|Y^{\psi^{{\theta}},m}-\sE[Y^{\psi^{{\theta}}}]\|_2
\big),
\end{split}
\end{equation}
where  the last inequality
follows from  the assumption $\|(\sE[V^{\psi^{{\theta}}}])^{-1}\|_2\le C_1$.

We now 
estimate each term in the right-hand side of \eqref{eq:est},
and denote by $C$ a generic constant independent of $\theta\in \Theta, \delta\in (0,1/2), m\in \sN$.
By Theorem 
\ref{thm:concentration_ls},
 with probability at least $1-2\delta$,
 $\|V^{\psi^{{\theta}},m}-\sE[V^{\psi^{{\theta}}}]\|_2 \leq \delta_m$
 and 
 $\|Y^{\psi^{{\theta}},m}-\sE[Y^{\psi^{{\theta}}}]\|_2
 \leq \delta_m$,
 with the constant $\delta_m$ given by 
 \bb\label{eq:delta_m}
 \delta_m\coloneqq
 C\max \left\{
\bigg(\frac{-\ln\delta  }{m} \bigg)^{\frac{1}{2}}, \frac{-\ln\delta  }{m} \right\}.
 \ee

 Let 
  $m$ be a sufficiently large constant satisfying 
 $\delta_m+1/m\le 1/(2C_1)$,
 where 
 $C_1$ is the constant 
such that   $\|(\sE[V^{\psi^{{\theta}}}])^{-1}\|_2\le C_1$ for all $\theta\in \Theta$.
Then 
 with probability at least $1-2\delta$, 
$ \|V^{\psi^{{\theta}},m}-\sE[V^{\psi^{{\theta}}}]+\tfrac{1}{m}I\|_2
 \le \delta_m+ \frac{1}{m}\le \tfrac{1}{2C_1}$,
 which in turn yields
 \begin{align*}
\lambda_{\min}( V^{\psi^{{\theta}},m}+\tfrac{1}{m})
\ge 
\lambda_{\min}( \sE[V^{\psi^{{\theta}}}])
-\|V^{\psi^{{\theta}},m}-\sE[V^{\psi^{{\theta}}}]+\tfrac{1}{m}I\|_2
\ge \tfrac{1}{2C_1},
 \end{align*}
 or equivalently $\|(V^{\psi^{{\theta}},m}+\tfrac{1}{m})^{-1}\|_2\leq 2C_1$.  
Moreover,  the continuity of $\sR^{(n+d)\t n}\ni \theta\mapsto \sE[Y^{\psi^{{\theta}}}]\in \sR$ implies
$\|Y^{\psi^{{\theta}},m}\|_2\le \|\sE[Y^{\psi^{{\theta}}}]\|_2+\|Y^{\psi^{{\theta}},m}-\sE[Y^{\psi^{{\theta}}}]\|_2
\le C+\|Y^{\psi^{{\theta}},m}-\sE[Y^{\psi^{{\theta}}}]\|_2$.
Hence,
by   \eqref{eq:est}, 
\begin{equation*}
\begin{split}
&|\hat{\theta}-\theta^\star|
\\
&\le
C\big(
(1+\|Y^{\psi^{{\theta}},m}-\sE[Y^{\psi^{{\theta}}}]\|_2)
\|V^{\psi^{{\theta}},m}-\sE[V^{\psi^{{\theta}}}]+\tfrac{1}{m}I\|_2
+\|Y^{\psi^{{\theta}},m}-\sE[Y^{\psi^{{\theta}}}]\|_2
\big)
\\
&\le 
C\left( (\delta_m+ \tfrac{1}{m})(1+ \delta_m)+\delta_m\right)
\le
C\left( \delta_m+\delta_m^2+ \tfrac{1}{m} \right).
\end{split}
\end{equation*}
Substituting \eqref{eq:delta_m} into the above estimate yields
the desired estimate \eqref{eq:parameter_estimation}.
As $\delta\in (0,1/2)$,
it is clear that 
 $\delta_m+1/m\le 1/(2C_1)$ is satisfied 
 for all $m\ge \bar{C}_1(-\ln \delta)$, with a sufficiently large $C_1$.
\end{proof}

\paragraph{Step 3: Proof of Theorem \ref{thm:regret_continuous}.} 

The following proposition
shows 
that for any given $\theta=(A,B)^\top\in \sR^{(n+d)\t n}$,
the full row rank of $K^\theta$ is  equivalent to
the well-definedness of 
 \eqref{rewrTheta} for all $\theta'$ sufficiently close to $\theta$.

\begin{Proposition}\label{prop:non_degenerate}
Suppose (H.\ref{assum:ls}\ref{assum:pd}) holds.
For each $\theta\in \sR^{(n+d)\t n}$,
let $V^{\psi^{\theta}}$  be  defined in
\eqref{eq:rv_ls_psi_conts}.
Then
for any 
$\theta=(A,B)^\top \in \sR^{(n+d)\t n}$,
 the following  properties are equivalent:
\begin{enumerate}[(1)]
\item\label{item:B_full_rank}
$\{v\in \sR^d\mid (K^\theta_t)^\top v=0,
\;
\fa t\in [0,T]\}=\{0\}$,
with $K^\theta$
 defined in \eqref{K2nonoise};
 \item\label{item:theta_B}
$\sE[V^{\psi^{{\theta}}}]\in \sS_+^{n+d}$;
\item
\label{item:theta_eps}
there exist $\lambda_0, \eps>0$ such that 
 $
\lambda_{\min}( \sE[V^{\psi^{{\theta}'}}])\ge \lambda_0$
for all
${\theta}'\in \Phi_\eps\coloneqq \{\theta'\in \sR^{(n+d)\t n}\mid |{\theta}'-\theta|\le \eps\}$,
where
$\lambda_{\min}(Z)$ is the minimum eigenvalue of 
 $Z\in \sS^{n+d}_0$.
\end{enumerate}

\end{Proposition}
\begin{proof}
For 
$\ref{item:B_full_rank}
 \Longrightarrow
\ref{item:theta_B}$:
By \eqref{eq:rv_ls_psi_conts}, 
$\sE[V^{\psi^{\theta}}]\in \sS^{n+d}_+$ if and only if 
there exists no nonzero $v\in \sR^{n+d}$ such that
\bb\label{eq:non_degenerate_stochastic}
 \sE\left[\int_0^T 
v^\top Z^{\psi^{\theta}}_{t} (Z^{\psi^{\theta}}_{t})^\top v\,\d t
\right]
=
\int_0^T 
v^\top 
\begin{pmatrix}
I \\
K^{\theta}_t
\end{pmatrix}
 \sE\left[ 
 X^{\psi^{\theta}}_{t}
(X^{\psi^{\theta}}_{t})^\top
\right]
\begin{pmatrix}
I  &
(K^{\theta}_t)^\top
\end{pmatrix}
 v\,\d t
=0,
\ee
where we applied Fubini's theorem for the first identity. 
By \eqref{eq:state_psi_theta},
$X^{\psi^{\theta}}_{t}=
\Phi^{\theta}_t \big(x_0+\int_0^t(\Phi^{\theta}_s)^{-1}\,\d W_s\big)
$ for all $t\in [0,T]$,
where 
$\Phi^{\theta}\in C([0,T]; \sR^{n\t n})$ is the fundamental  
solution of  
$\d \Phi^{\theta}_t=(A^\star+B^\star K^{\theta}_t)\Phi^{\theta}_t \,\d t$, 
$K^\theta_t=-R^{-1}B^\top P^\theta_t$ for all $t\in[0,T]$,
and 
$P^\theta$ satisfies
\eqref{riccati2}.
Hence, 
$$
\sE\left[ 
 X^{\psi^{\theta}}_{t}
(X^{\psi^{\theta}}_{t})^\top
\right]
=
\Phi^{\theta}_t
\left(
x_0x_0^\top+
\int_0^t 
(\Phi^{\theta}_s)^{-1}
((\Phi^{\theta}_s)^{-1})^\top
\,\d s 
\right)
(\Phi^{\theta}_t)^\top
\in \sS^n_0,
\q \fa t\in[0,T].
$$
Then by
\eqref{eq:non_degenerate_stochastic}
and the continuity of
  $t\mapsto \sE\left[ 
 X^{\psi^{\theta}}_{t}
(X^{\psi^{\theta}}_{t})^\top
\right]$ and 
$t\mapsto K^{\theta}_t$, 
$\sE[V^{\psi^{\theta}}]\in \sS^{n+d}_+$ if and only if 
there exists no nonzero  $v\in \sR^{n+d}$ such that
$$
v^\top 
\begin{pmatrix}
I \\
K^{\theta}_t
\end{pmatrix}
 \Phi^{\theta}_t
\left(
x_0x_0^\top+
\int_0^t 
(\Phi^{\theta}_s)^{-1}
((\Phi^{\theta}_s)^{-1})^\top
\,\d s 
\right)
(\Phi^{\theta}_t)^\top
\begin{pmatrix}
I  &
(K^{\theta}_t)^\top
\end{pmatrix}
 v=0,
 \q \fa t\in [0,T],
 $$
 where $I$ is the $n\t n$ identity matrix.
 One can easily deduce  by 
 the invertibility of $(\Phi^{\theta}_t)^{-1}$ for all $t\in [0,T]$  that 
 $\int_0^t 
(\Phi^{\theta}_s)^{-1}
((\Phi^{\theta}_s)^{-1})^\top
\,\d s\in \sS^n_+$ for all $t>0$, 
 which subsequently shows that  
 $\sE[V^{\psi^{\theta}}]\in \sS^{n+d}_+$ if and only if 
there exists no nonzero  $\tilde{v}\in \sR^{n+d}$ such that
$\begin{pmatrix}
I  &
(K^{\theta}_t)^\top
\end{pmatrix}
 \tilde{v}=0
 $
 for all $t\in [0,T]$.
 {\color{black}
 Now let us denote  without loss of generality that
 $\tilde{v}=\begin{psmallmatrix}
 u\\ v
 \end{psmallmatrix}$
 for some $u\in \sR^n$ and $v\in \sR^d$.
 Then the above derivation shows that 
  $\sE[V^{\psi^{\theta}}]\in \sS^{n+d}_+$ is equivalent to the following statement:
  \bb\label{eq:equivalence_uv}
 \textnormal{if  $u\in \sR^n$ and $v\in \sR^d$ satisfy
 $u+(K^{\theta}_t)^\top v=0$
for all $t\in [0,T]$,
then $u=0$ and $v=0$.
}
\ee
By \eqref{K2nonoise}, $(K^{\theta}_t)^\top=-P^{\theta}_t B R^{-1}$ for all $t\in [0,T]$
 and $P^\theta_T=0$, 
 implying that $K^\theta_T=0$.
 Then \eqref{eq:equivalence_uv} can be rewritten as:
   \begin{equation*}
 \textnormal{if  $v\in \sR^d$ satisfies
 $(K^{\theta}_t )^\top v=0$
for all $t\in [0,T]$,
then  $v=0$.}
\end{equation*}
}%

For 
 $ \ref{item:theta_B}
 \Longleftrightarrow
 \ref{item:theta_eps}
 $:
 Item \ref{item:theta_eps}
clearly implies Item \ref{item:theta_B}.
On the other hand, 
for any given $\theta,\theta'\in \sR^{(n+d)\t n}$, 
$$
\d(X^{\psi^{\theta}}- X^{\psi^{\theta'}})_t
=
\Big((A^\star+B^\star K^{\theta}_t)(X^{\psi^{\theta}}- X^{\psi^{\theta'}})_t
+B^\star (K^{\theta}_t-K^{\theta'}_t)  X^{\psi^{\theta'}}_t
\Big)\,\d t.
$$
Then, we can easily deduce
from the continuity of $t\mapsto K^\theta$
(see Lemma \ref{lemma:a_priori_riccati})
  that 
$\sR^{(n+d)\t n}\ni \theta\mapsto Z^{\psi^\theta}\in \cH^2(\sR^{(n+d)\t n})$
is continuous,
which 
implies the continuity of  
$
\sR^{(n+d)\t n}\ni \theta
\mapsto
V^{\psi^\theta}
= \sE\big[\int_0^T 
Z^{\psi^\theta}_{t}(Z^{\psi^\theta}_{t})^\top\,\d t\big]
\in \sS^{n+d}_0$.
Hence, by the continuity of the minimum eigenvalue function,
we 
can conclude
 Item \ref{item:theta_B}
 from
 Item \ref{item:theta_eps}.
 \end{proof}

{\color{black}
The following proposition provides sufficient conditions for 
the nondegeneracy of $K^\theta$.
\begin{Proposition}
\label{prop:non_degenerate_sufficient condition}
Let
$n,d\in \sN$, 
$\theta=(A,B)^\top \in \sR^{(n+d)\t n}$,
$Q\in \sS_0^n$ and $R\in \sS_+^d$.
\begin{enumerate}[(1)]
\item\label{item:sufficient_arbitraryT}
For all $T>0$, if 
 $B^\top Q B\in \sS^d_+$,
 then $\{v\in \sR^d\mid (K^\theta_t)^\top v=0,
\;
\fa t\in [0,T]\}=\{0\}$.

\item
\label{item:sufficient_largeT}
Assume that  the algebraic Riccati equation 
$
A^\top P+PA  -P (BR^{-1}B^\top )P+Q=0
$
admits a 
unique  maximal  solution
$P_\infty\in \sS^n_+$.
Let $K_\infty=-R^{-1}B^\top P_\infty$, and
for each $T>0$, let $P^{(T)}\in C([0,T];\sS_0^n)$ be defined in \eqref{riccati2}.
Assume that $\lim_{T\to \infty} P^{(T)}_0=P_\infty$ and $K_\infty(K_\infty)^\top\in \sS_+^d $.
Then there exists $T_0>0$, such that for all $T\ge T_0$,
$\{v\in \sR^d\mid (K^\theta_t)^\top v=0,
\;
\fa t\in [0,T]\}=\{0\}$.
\end{enumerate}

\end{Proposition}

\begin{proof}
To prove Item \ref{item:sufficient_arbitraryT},
suppose that $B^\top Q B\in \sS_+^n$ and
 $v\in \sR^d$ such that $(K^{\theta}_t)^\top v=
 -P^{\theta}_t BR^{-1} v=0$ for all $t\in [0,T]$,
 with $P^\theta$ defined in \eqref{riccati2}.
 Setting $u=R^{-1}v$,
 right multiplying \eqref{riccati2} by $B u $,
 and 
 left multiplying \eqref{riccati2} by $u^\top B^\top $ shows 
 \begin{equation*}
u^\top B^\top (\tfrac{\d}{\d t} P^\theta_t) Bu+ A^\top P^\theta_tBu + u^\top B^\top P^\theta_tA Bu - u^\top B^\top P^\theta_t B R^{-1}B^\top P^\theta_tBu  + u^\top B^\top QBu=0.
\end{equation*}
As $P^\theta_t Bu=0$ for all $t\in (0,T)$,
$u^\top B^\top (\frac{\d }{\d t} P^{\theta}_t) B u=
u^\top   B^\top P^{\theta}_t=
0$
for all $t\in (0,T)$,
and hence 
$u^\top   B^\top QBu=0$.
The assumption of $B^\top QB\in \sS^n_+$ then gives
$u=R^{-1}v=0$, which along with the   invertibility of $R^{-1}$
shows that $v=0$.

To prove Item 
\ref{item:sufficient_largeT}, observe that $\lim_{T\to \infty} (-R^{-1} B^\top P_0^{(T)})=K_\infty$.
As $\lambda_{\min} (K_\infty(K_\infty)^\top)>0$, there exists $T_0>0$
such that for all $T\ge T_0$,
$
\lambda_{\min} \left(
(-R^{-1} B^\top P_0^{(T)})
(-R^{-1} B^\top P_0^{(T)})^\top
\right)>0
$. Fix $T\ge T_0$ and consider 
$v\in \sR^d$ such that $(K^{\theta}_t)^\top v
=0$ for all $t\in [0,T]$.
Then  the definitions of $K^\theta$  and $P^{(T)}$
imply the invertibility  of 
 $K^\theta_{0} (K^\theta_{0})^\top$,
 which yields
 $v=(K^\theta_{0} (K^\theta_{0})^\top)^{-1}K^\theta_{0} (K^\theta_{0})^\top v=0$.
\end{proof}

}

Now we are ready for the proof of Theorem \ref{thm:regret_continuous}.

\begin{proof}[Proof of Theorem \ref{thm:regret_continuous}]
As  
 (H.\ref{assum:ls}\ref{assum:non_degenerate}) holds with
 $\theta^\star$ and $\theta_0$,
we can obtain from 
Proposition \ref{prop:non_degenerate} that,
 there exist $C_1,\eps>0$
such that 
for all $\theta\in 
\Phi_\eps\coloneqq \{\theta\mid \sR^{(n+d)\t n}\mid |\theta-\theta^\star|\le \eps\}\cup \{\theta_0\}$,
we have 
$\|(\sE[V^{\psi^{{\theta}}}])^{-1}\|_2\le C_1$.
Then by Proposition \ref{theta_conc},
 there exist constants  $\bar{C}_1,\bar{C}_2\ge  1$,
such that 
for all
$\theta\in \Theta_\eps$ and
 $\delta'\in (0,1/2)$, if 
$m\geq \bar{C}_1(-\ln\delta' )$, 
 then  with probability at least $1-2\delta'$, 
\bb\label{eq:Theta_eps_error}
|\hat{\theta}-\theta^\star|\leq \bar{C}_2\left(\sqrt{\frac{-\ln\delta'}{m}}+\frac{-\ln\delta' }{m}+\frac{(-\ln\delta')^{2}}{m^2}\right),
\ee
where 
$\hat{\theta}$ denotes the  right-hand side of \eqref{reg_ls}
 with the control $\psi^\theta$.
In the following,
we   fix   
$\delta\in (0,3/\pi^2)$ and 
 $C\ge C_0$,
with 
the constant
${C}_0\in (0,\infty)$ satisfying 
$$
{C}_0\ge \bar{C}_1
\bigg(
\sup_{\ell\in \sN\cup\{0\},\delta\in (0,{3}/{\pi^2})}\frac{-\ln (\delta/(\ell+1)^{2})}{2^\ell (-\ln \delta)}
\bigg)
\bigg/ \min\bigg\{\left(\frac{\eps}{3\bar{C}_2}\right)^2,1\bigg\},
$$
let $m_0= C(-\ln \delta)$,
and 
for each  $\ell\in \sN\cup\{0\}$,
 let   $\delta_\ell=\delta/ (\ell+1)^{2}$,
 $m_\ell=2^\ell m_0$, 
and let
 $\theta_{\ell+1}$ be generated by
   \eqref{reg_ls} with   $m=m_\ell$ and $\theta=\theta_\ell$.
Note that the choices of $C_0,m_\ell,\delta_\ell$
 ensure that 
 $m_\ell\ge \bar{C}_1(-\ln \delta_{\ell})$,
 and 
\bb\label{eq:m_0_hatC_0}
\bar{C}_2\left(\sqrt{\frac{-\ln\delta_\ell}{m_\ell}}+\frac{(-\ln\delta_\ell)}{m_\ell}+\frac{(-\ln\delta_\ell)^{2}}{m^2_\ell}\right)
\le
3\bar{C}_2\sqrt{\frac{-\ln\delta_\ell}{m_\ell}}
\le \eps,
\q \fa \ell \in \sN\cup\{0\}.
\ee

We now prove    
with   probability  at least $1-2\sum_{\ell=0}^{\infty}\delta_\ell=1-\frac{\pi^2}{3}\delta$,
\bb\label{eq:theta_ell}
|{\theta}_{\ell+1}-\theta^\star|\leq
 \bar{C}_2\bigg(\sqrt{\frac{-\ln\delta_\ell}{m_\ell}}+\frac{(-\ln\delta_\ell)}{m_\ell}+\frac{(-\ln\delta_\ell)^{2}}{m_\ell^2}\bigg),
 \q \fa  \ell\in \sN\cup\{0\}.
 \ee
Let us consider the induction statement
for each $k\in \sN\cup\{0\}$: 
  with probability  at least $1-2\sum_{\ell=0}^{k}\delta_\ell$,
   \eqref{eq:theta_ell} holds for all   $\ell=0,\ldots, k$.
The fact that  $\theta_0\in \Theta_\eps$
and \eqref{eq:Theta_eps_error}
yields 
  the induction statement   for $k=0$.
Now suppose that the induction statement holds for some $k\in  \sN\cup\{0\}$.
Then 
  the induction hypothesis
and  \eqref{eq:m_0_hatC_0}
ensure that   $|\theta_{\ell}-\theta^\star|\le \eps$
for all $\ell=1,\ldots,k+1$
 (and hence $\theta_{k+1} \in \Theta_\eps$)
  with probability  at least $1-2\sum_{\ell=0}^{k}\delta_\ell$.
Conditioning on this event,
we can apply
\eqref{eq:Theta_eps_error} with 
$\theta=\theta_{k+1}$,
$\delta'=\delta_{k+1}<1/2$ and 
$m=m_{k+1}\ge \bar{C}_1(-\ln\delta_{k+1} )$,
and deduce 
  with probability   
 at least $1-2\delta_{k+1}$
 that
  \eqref{eq:theta_ell} holds for the index $\ell=k+1$.
  Combining this with the induction hypothesis
  yields 
    \eqref{eq:theta_ell} holds for the indices $\ell=0,\ldots, k+1$,
with probability  at least $1-2\sum_{\ell=0}^{k+1}\delta_\ell$.

Observe that 
for all $i\in \sN$, 
 Algorithm \ref{alg:conts_ls}
generates
the $i$-th trajectory  with control $\psi^{\theta_\ell}$
if $i\in (\sum_{j=0}^{\ell-1} m_j,\sum_{j=0}^{\ell}m_j]=( m_0(2^\ell-1), m_0(2^{\ell+1}-1)]$ with some $\ell\in \sN\cup \{0\}$.
Then
conditioning on the event \eqref{eq:theta_ell},
we can obtain from 
 Proposition \ref{prop:performance_gap} that, 
 for all $M\in \sN$,
\begin{align}\label{eq:R_N_conts}
\begin{split}
R(M)&
\le \sum_{\ell=0}^{\lceil \log_2(\frac{M}{m_0}+1)\rceil-1 } m_\ell 
\Big(
J^{\theta^\star}(U^{\psi^{\theta_\ell}})-J^{\theta^\star}(U^{\theta^\star})
\Big)
\leq
C' \sum_{\ell=0}^{\lceil \log_2(\frac{M}{m_0}+1)\rceil-1 } m_\ell |\theta_\ell-\theta^\star|^2
\\
&\le C'm_0+C'\sum_{\ell=0}^{\lceil \log_2(\frac{M}{m_0}+1)\rceil-1 }  (-\ln\delta_\ell)  \Big(1+\frac{-\ln\delta_\ell}{m_\ell}+\frac{(-\ln\delta_\ell)^{3}}{m_\ell^3}\Big)
\\
&\le
C'(-\ln \delta)+C'\sum_{\ell=1}^{\lceil \log_2{M}\rceil } \bigg(2\ln \ell -\ln \delta\bigg)
\le 
C'
\left(
(\ln M)(\ln\ln M)+(\ln M)(-\ln \delta)
\right),
\end{split}
\end{align}
with a constant $C'$ independent of $M$ and $\delta$,
where we have used  
$\sum_{\ell=1}^n \ln \ell=\ln (n!)\le  C'n\ln n$ due to Stirling's formula.
\end{proof}

\subsection{Regret analysis of discrete-time least-squares algorithm}
\label{sec:regret_discrete}

This section is devoted to the proof of 
 Theorem \ref{thm:regret_discrete}.
 The main step is  similar to the proof of Theorem   \ref{thm:regret_continuous}
in Section \ref{sec:regret_continuous}.
However, one needs to quantity the precise impact of 
the piecewise constant policies 
and the discrete-time observations 
on the performance gap and the parameter estimation error.

\paragraph{Step 1: Analysis of the performance gap.}

The following proposition shows the performance gap between
applying a piecewise constant  feedback 
control   for an incorrect model
 and a continuous-time 
feedback control for the true model 
scales quadratically with respect to the stepsize
and the parameter errors.

\begin{Proposition}\label{prop:performance_gap_discrete}
Suppose (H.\ref{assum:ls}\ref{assum:pd}) holds
and 
let  $\Theta$ be a bounded subset of 
$\sR^{(n+d)\t n}$.
For each $\theta\in\Theta$ and $N\in \sN$, 
let 
$\psi^{\theta,\tau}$
be defined in \eqref{eq:feedback_theta_tau}
with stepsize $\tau=T/N$, 
let
$X^{\psi^{\theta,\tau}}$
be  the state process
 associated with    $\psi^{ \theta,\tau}$
 (cf.~\eqref{eq:sde_theta_pi}),
let $\psi^{\theta^\star}$ be defined in
\eqref{K1},
and let
$X^{{\theta^\star}}$
be  the state process
 associated with   
$\psi^{ \theta^\star}$ (cf.~\eqref{eq:state_psi_theta}).
Then there exists   $C>0$ such that 
\bb\label{item:performance_gap_discrete}
|J^{\theta^\star}(U^{\psi^{\theta,\tau}})-J^{\theta^\star}(U^{\theta^\star})|
\le C(N^{-2}+|\theta-\theta^\star|^2),
\q \fa \theta\in \Theta, N\in \sN,
\ee
where $U^{\psi^{\theta,\tau}}_t=\psi^{\theta,\tau}(t,X^{\psi^{\theta,\tau}}_t)$ and 
$U^{\theta^\star}_t=\psi^{\theta^\star}(t, X^{{\theta^\star}}_t)$ for all $t\in [0,T]$,
and  $J^{\theta^\star}$ is defined in \eqref{LQ}.

\end{Proposition}

{\color{black}
\begin{proof}
Let us fix 
 $\theta\in \Theta$ and $N\in \sN$.
By
applying  Proposition \ref{prop:taylor_expansion}
with $U=U^{\psi^{\theta,\tau}}$, 
\begin{align}\label{eq:quadratic_gap_feedback_discrete}
\begin{split}
&J^{\theta^\star}(U^{\psi^{\theta,\tau}})-J^{\theta^\star}(U^{\theta^\star})
\\
&
\leq  \|Q\|_2\|X^{\theta^\star, U^{\psi^{\theta,\tau}}}-X^{{\theta^\star}}\|^2_{\cH^2(\sR^n)}+\|R\|_2
\|U^{\psi^{\theta,\tau}}-U^{\theta^\star}\|^2_{\cH^2(\sR^d)},
\\
&\leq  \|Q\|_2\|X^{{\psi^{\theta,\tau}}}-X^{\psi^{\theta^\star}}\|^2_{\cH^2(\sR^n)}+\|R\|_2
\|\psi^{\theta,\tau}(\cdot, X^{\psi^{\theta,\tau}}_\cdot)-\psi^{\theta^\star}(\cdot, X^{\psi^{\theta^\star}}_\cdot)\|^2_{\cH^2(\sR^d)},
\end{split}
\end{align}
where 
the last inequality used the fact that $X^{\theta^\star, U^{\psi^{\theta,\tau}}}=X^{{\psi^{\theta,\tau}}}$
(see  \eqref{lin_dyn}),
and the definitions of 
 $U^{\psi^{\theta,\tau}}$ and 
$U^{\theta^\star}$.

We then prove that there exists a  constant  $C$, independent of $\theta, N$, such that 
$$
\|X^{{\psi^{\theta,\tau}}}-X^{\psi^{\theta^\star}}\|_{\cH^2(\sR^n)}+
\|\psi^{\theta,\tau}(\cdot, X^{\psi^{\theta,\tau}}_\cdot)-\psi^{\theta^\star}(\cdot, X^{\psi^{\theta^\star}}_\cdot)\|_{\cH^2(\sR^d)}
\le C(N^{-1}+|\theta-\theta^\star|).
$$ 
By setting  $\delta X=X^{\theta^{\star}}-X^{\psi^{\theta,\tau}}$,
 we obtain from
 \eqref{eq:state_psi_theta} and
  \eqref{eq:sde_theta_pi} that
\begin{align}
\label{eq:delta X_pi}
\d  \delta X_t
=(A^\star \delta X_t+B^\star K^{\theta^\star}_t \delta X_t+(K^{\theta^\star}_t -K^{\theta,\tau}_t)X^{\psi^{\theta,\tau}}_t)\, \d t,
\q t\in [0,T]; \q \delta X_0=0.
\end{align}
Since $\|P^{\theta^\star}\|_{C([0,T];\sR^{n\t n})}\le C$
and $K^{\theta^\star}_t=-R^{-1}B^\top P^{{\theta^\star}}_t$ for all $t\in [0,T]$, 
$\|K^{\theta^\star}\|_{C([0,T];\sR^{d\t n})}\le C$. 
Moreover, 
by  $\|P^{\theta,\tau}_{t_i}\|_2\le C$
 for all $i=0,\ldots,N$  
(see Proposition \ref{prop:riccati_conv_rate})
and 
 \eqref{eq:K_pi},
 we have 
 $\|K^{\theta,\tau}_t\|_2\le C$
 for all $t\in [0,T]$,
 which along with a  moment estimate of \eqref{eq:sde_theta_pi}
 yields $\|X^{\psi^{\theta,\tau}}\|_{\cS^2(\sR^n)}\le C$.
Thus,   by applying
Gronwall's inequality to \eqref{eq:delta X_pi}, 
Lemma \ref{lemma:a_priori_riccati}
and
 Proposition    \ref{prop:riccati_conv_rate},
 for all 
$ \theta\in \Theta$ and $N\in \sN$,
\begin{align}\label{eq:X_L2_rate}
\begin{split}
&\|X^{\theta^\star}-X^{\psi^{\theta,\tau}}\|_{\cH^2(\sR^n)}
\le 
C\|X^{\theta^\star}-X^{\psi^{\theta,\tau}}\|_{\cS^2(\sR^n)}
\\
&\le C\|(K^{\theta^\star}_t -K^{\theta,\tau}_t)X^{\psi^{\theta,\tau}}\|_{\cH^2(\sR^d)}
\le C\max_{t\in [0,T]}\|K^{\theta^\star}_t -K^{\theta,\tau}_t\|_2
\\
&\le
C\max_{t\in [0,T]}(\|K^{\theta}_t -K^{\theta,\tau}_t\|_2+\|K^{\theta^\star}_t -K^{\theta}_t\|_2)
\le 
 C(N^{-1}+|\theta-\theta^\star|).
\end{split}
\end{align}
The above inequality further implies  
\begin{align*}
& \|\psi^{\theta,\tau}(\cdot, X^{\psi^{\theta,\tau}}_\cdot)-\psi^{\theta^\star}(\cdot, X^{\psi^{\theta^\star}}_\cdot)\|_{\cH^2(\sR^d)}
=\|K^{\theta,\tau}_\cdot X^{\psi^{\theta,\tau}}_\cdot-K^{\theta^\star}_\cdot  X^{\psi^{\theta^\star}}_\cdot\|_{\cH^2(\sR^d)}
\\
&\le 
\|( K^{\theta,\tau}_\cdot -K^{\theta^\star}_\cdot ) X^{\psi^{\theta,\tau}}_\cdot\|_{\cH^2(\sR^d)}
+
\|
K^{\theta^\star}_\cdot  (X^{\psi^{\theta,\tau}}_\cdot - X^{\psi^{\theta^\star}}_\cdot)\|_{\cH^2(\sR^d)}
\\
&\le   
\|K^{\theta^\star}-K^{\theta,\tau}\|_{C([0,T;\sR^{d\t n})} \| X^{\psi^{\theta,\tau}}\|_{\cH^2(\sR^{n})}
+\|
K^{\theta^\star}\|_{C([0,T;\sR^{d\t n})} \| X^{\psi^{\theta,\tau}} - X^{\psi^{\theta^\star}}\|_{\cH^2(\sR^n)}
\\
&\le C(N^{-1}+|\theta-\theta^\star|),
\q \fa \theta\in \Theta, N\in \sN,
\end{align*}
which along with \eqref{eq:quadratic_gap_feedback_discrete} finishes the desired estimate.
\end{proof}
}

\paragraph{Step 2: Error bound for parameter estimation.}

The following lemma shows that
the difference between 
the expectations of 
$(V^{\psi^{\theta,\tau},\tau},Y^{\psi^{\theta,\tau},\tau})$
and 
of $(V^{\psi^{\theta,\tau}},Y^{\psi^{\theta,\tau}})$
scales linearly with respect to 
the stepsize.

\begin{Lemma}\label{lemma:discrete_expectation}
Suppose (H.\ref{assum:ls}\ref{assum:pd}) holds
and 
let  $\Theta$ be a bounded subset of 
$\sR^{(n+d)\t n}$.
For each $\theta \in\Theta$ and $N\in \sN$, 
let $\tau=T/N$,
let
$\psi^{\theta,\tau}$
be defined in \eqref{eq:feedback_theta_tau},
let 
$V^{\psi^{\theta,\tau}},Y^{\psi^{\theta,\tau}}$
be defined in
\eqref{eq:rv_ls_psi_conts},
and
 let 
$V^{\psi^{\theta,\tau},\tau},Y^{\psi^{\theta,\tau},\tau}$
be defined in
\eqref{eq:rv_ls_psi_discrete}.
Then there exists a constant $C$ such that 
$$
|\sE[V^{\psi^{\theta,\tau},\tau}-V^{\psi^{\theta,\tau}}]|
+
|\sE[Y^{\psi^{\theta,\tau},\tau}-Y^{\psi^{\theta,\tau}}]|
\le CN^{-1},
\q \fa 
\theta\in \Theta, N\in \sN.
$$

\end{Lemma}
\begin{proof}
%
%
By
\eqref{eq:sde_theta_pi}, 
we have 
for all $i=0,\ldots, N-1$,
$ X^{\psi^{\theta,\tau}}_{t_{i+1}}-X^{\psi^{\theta,\tau}}_{t_{i}}=\int_{t_i}^{t_{i+1}}
(\theta^\star)^\top Z^{\psi^{\theta,\tau}}_t\,\d t+W_{t_{i+1}}-W_{t_{i}}$,
which implies 
\begin{align*}
\sE[V^{\psi^{\theta,\tau}}-V^{\psi^{\theta,\tau},\tau}]
&=
\sum_{i=0}^{N-1}
\int_{t_i}^{t_{i+1}} 
\sE[
Z^{\psi^{\theta,\tau}}_{t}(Z^{\psi^{\theta,\tau}}_{t})^\top
-
Z^{\psi^{\theta,\tau}}_{t_i}(Z^{\psi^{\theta,\tau}}_{t_i})^\top
]
\,\d t,
\\
\sE[Y^{\psi^{\theta,\tau}}-Y^{\psi^{\theta,\tau},\tau}]
&=
\sum_{i=0}^{N-1}
\int_{t_i}^{t_{i+1}} 
\sE[
(Z^{\psi^{\theta,\tau}}_{t}
-Z^{\psi^{\theta,\tau}}_{t_i})
(Z^{\psi^{\theta,\tau}}_{t})^\top
\theta^\star
]
\,\d t.
\end{align*}
Hence it suffices to
 prove that 
$|\sE[
Z^{\psi^{\theta,\tau}}_{t}(Z^{\psi^{\theta,\tau}}_{t})^\top
-
Z^{\psi^{\theta,\tau}}_{t_i}(Z^{\psi^{\theta,\tau}}_{t_i})^\top
]|\le CN^{-1}$
and
$|\sE[
(Z^{\psi^{\theta,\tau}}_{t}
-Z^{\psi^{\theta,\tau}}_{t_i})
(Z^{\psi^{\theta,\tau}}_{t})^\top]
|\le CN^{-1}$
for all $t\in [t_i,t_{i+1}]$ and $i=0,\ldots, N-1$.

Let us fix $i=0,\ldots, N-1$ and $t\in [t_i,t_{i+1}]$.
In the following, we shall omit the superscripts  of 
  $X^{\psi^{\theta,\tau}}$ and $Z^{\psi^{\theta,\tau}}$
  if no confusion occurs. 
  As $t\in [t_i,t_{i+1}]$,
by  \eqref{eq:sde_theta_pi},
we have 
$X_t=e^{Lt}X_{t_i}+\int_{t_i}^te^{L(t-s)}\,\d W_s$
with $L\coloneqq A^\star+B^\star K^{\theta,\tau}_{t_i}$.
Thus, 
\begin{align*}
&
X_tX^\top_t-X_{t_i}X^\top_{t_i}
\\
&=(X_t-X_{t_i}+X_{t_i})(X_t-X_{t_i}+X_{t_i})^\top-X_{t_i}X^\top_{t_i}
\\
&=
(X_t-X_{t_i})(X_t-X_{t_i})^\top
+X_{t_i}(X_t-X_{t_i})^\top
+(X_t-X_{t_i})X_{t_i}^\top
\\
&=
\bigg((e^{Lt}-I)X_{t_i}+\int_{t_i}^te^{L(t-s)}\,\d W_s\bigg)
\bigg((e^{Lt}-I)X_{t_i}+\int_{t_i}^te^{L(t-s)}\,\d W_s\bigg)^\top
\\
&\q +X_{t_i}\bigg((e^{Lt}-I)X_{t_i}+\int_{t_i}^te^{L(t-s)}\,\d W_s\bigg)^\top
+\bigg((e^{Lt}-I)X_{t_i}+\int_{t_i}^te^{L(t-s)}\,\d W_s\bigg)X_{t_i}^\top.
\end{align*}
By taking  expectations of both sides of the above identity,  the martingale property of
the
 It\^{o} integral,
and the It\^{o} isometry, 
\begin{align*}
\sE[X_tX^\top_t-X_{t_i}X^\top_{t_i}]
&=
(e^{Lt}-I)\sE[X_{t_i}X_{t_i}^\top](e^{L^\top t}-I)
+\int_{t_i}^te^{L(t-s)}e^{L^\top (t-s)}\,\d s
\\
&\q +\sE[X_{t_i}X_{t_i}^\top] (e^{L^\top t}-I)
+(e^{Lt}-I)\sE[X_{t_i}X_{t_i}^\top]
\le 
C(t-t_i),
\end{align*}
where 
 the last inequality follows from 
$\|X^{\psi^{\theta,\tau}}\|_{\cS^2(\sR^n)}\le C$.
Since $\psi^{\theta,\tau}(t,X^{\psi^{\theta,\tau}}_t)=K^{\theta,\tau}_{t_i}X^{\psi^{\theta,\tau}}_t$
and $\|K^{\theta,\tau}_{t_i}\|_{2}\le C$, one can easily show that 
$|\sE[
Z^{\psi^{\theta,\tau}}_{t}(Z^{\psi^{\theta,\tau}}_{t})^\top
-Z^{\psi^{\theta,\tau}}_{t_i}(Z^{\psi^{\theta,\tau}}_{t_i})^\top]
|
\le CN^{-1}$.
Furthermore,
by
$X^{\psi^{\theta,\tau}}_t=e^{Lt}X^{\psi^{\theta,\tau}}_{t_i}+\int_{t_i}^te^{L(t-s)}\,\d W_s$
and
 the identity  
$$
Z^{\psi^{\theta,\tau}}_{t}(Z^{\psi^{\theta,\tau}}_{t})^\top
-Z^{\psi^{\theta,\tau}}_{t_i}(Z^{\psi^{\theta,\tau}}_{t_i})^\top
=(Z^{\psi^{\theta,\tau}}_{t}
-Z^{\psi^{\theta,\tau}}_{t_i})
(Z^{\psi^{\theta,\tau}}_{t})^\top
+Z^{\psi^{\theta,\tau}}_{t_i}(Z^{\psi^{\theta,\tau}}_{t}-Z^{\psi^{\theta,\tau}}_{t_i})^\top,
$$
we can show  that
\begin{align*}
&|\sE[(Z^{\psi^{\theta,\tau}}_{t}
-Z^{\psi^{\theta,\tau}}_{t_i})
(Z^{\psi^{\theta,\tau}}_{t})^\top]|
\\
&\le 
|\sE[Z^{\psi^{\theta,\tau}}_{t}(Z^{\psi^{\theta,\tau}}_{t})^\top
-Z^{\psi^{\theta,\tau}}_{t_i}(Z^{\psi^{\theta,\tau}}_{t_i})^\top]|
+
|
\sE[Z^{\psi^{\theta,\tau}}_{t_i}(Z^{\psi^{\theta,\tau}}_{t}-Z^{\psi^{\theta,\tau}}_{t_i})^\top]
|
\\
&\le
C\left(N^{-1}+
\left|
\sE\left[Z^{\psi^{\theta,\tau}}_{t_i} (X^{\psi^{\theta,\tau}}_{t_i})^\top(e^{L^\top t}-I)
\begin{pmatrix}
I & (K^{\theta,\tau}_{t_i})^\top
\end{pmatrix}
\right]
\right|
\right)
\le CN^{-1},
\end{align*}
by  the uniform boundedness of $\|X^{\psi^{\theta,\tau}}\|_{\cS^2(\sR^n)}$ and $K^{\theta,\tau}$.
\end{proof}

\begin{Proposition}\label{theta_conc_discrete}

Suppose  (H.\ref{assum:ls}\ref{assum:pd}) holds,
and  let  $\Theta\subset \sR^{(n+d)\t n}$
such that 
 there exists
  $C_1>0$ 
satisfying
$\|(\sE[V^{\psi^{{\theta}}}])^{-1}\|_2\le C_1$
and $|\theta|\le C_1$
for all $\theta\in \Theta$,
with $V^{\psi^\theta}$  defined in  \eqref{eq:rv_ls_psi_conts}.
Then there exist constants  $\bar{C}_1,\bar{C}_2\ge 0$ and $n_0\in\sN$,
such that 
for all
$\theta\in \Theta$,
$N\in \sN\cap [n_0,\infty)$
 and
 $\delta\in (0,1/2)$, if 
$m\geq \bar{C}_1(-\ln\delta )$, 
 then   with probability at least $1-2\delta$, 
\bb\label{eq:parameter_estimation}
|\hat{\theta}-\theta^\star|\leq \bar{C}_2\bigg(\sqrt{\frac{-\ln\delta}{m}}+\frac{-\ln\delta }{m}+\frac{(-\ln\delta)^{2}}{m^2}
+\frac{1}{N}
\bigg),
\ee
where 
$\hat{\theta}$ denotes the  right-hand side of \eqref{reg_ls_discrete}
 with the control $\psi^{\theta,\tau}$
 and stepsize $\tau=T/N$.
\end{Proposition}

\begin{proof}
We first prove that there exists $n_0\in \sN$ such that for all $N\in \sN\cap [n_0,\infty)$ and $\theta\in \Theta$, 
$\|(\sE[V^{\psi^{{\theta,\tau}}}])^{-1}\|_2\le C$ for a  constant $C> 0$ independent of $\theta$ and $N$.
By \eqref{lin_dyn} and
\eqref{eq:sde_theta_pi}, we have
for all $\theta\in \Theta$ and $N\in \sN$,
$X^{\psi^{\theta}}_0= X^{\psi^{\theta},\tau}_0$ and 
$$
\d(X^{\psi^{\theta}}- X^{\psi^{\theta},\tau})_t
=
\Big((A^\star+B^\star K^{\theta}_t)(X^{\psi^{\theta}}- X^{\psi^{\theta},\tau})_t
+B^\star (K^{\theta}_t-K^{\theta,\tau}_t)  X^{\psi^{\theta},\tau}_t
\Big)\,\d t,
\q t\in [0,T].
$$
Proposition \ref{prop:riccati_conv_rate} shows 
$\|K^{\theta}_t-K^{\theta,\tau}_t\|_2\le CN^{-1}$
for all $t\in [0,T]$,
which along with Gronwall's inequality yields $\|X^{\psi^{\theta}}- X^{\psi^{\theta},\tau}\|_{\cS^2(\sR^n)}\le CN^{-1}$
for all $\theta\in \Theta$ and $N\in \sN$.
One can further prove that 
$\|U^{\psi^{\theta}}- U^{\psi^{\theta},\tau}\|_{\cS^2(\sR^d)}\le CN^{-1}$
with $U^{\psi^{\theta}}_t=\psi^{\theta}(t,X^{\psi^{\theta}}_t)$ and 
$U^{\psi^{\theta,\tau}}_t=\psi^{\theta,\tau}(t,X^{\psi^{\theta,\tau}}_t)$ for all $t\in [0,T]$.
Thus, we have 
$|\sE[V^{\psi^{{\theta}}}]-\sE[V^{\psi^{{\theta},\tau}}]|\le CN^{-1}$,
which along with $\|(\sE[V^{\psi^{{\theta}}}])^{-1}\|_2\le C_1$
implies a uniform bound of $\|(\sE[V^{\psi^{{\theta},\tau}}])^{-1}\|_2$ for all sufficiently large $N$.

Let us fix $N\in \sN\cap [n_0,\infty)$ and $\theta\in \Theta$ for the subsequent analysis.
The invertibility of $\sE[V^{\psi^{{\theta},\tau}}]$ implies that 
$\theta^\star=(\sE[V^{\psi^{{\theta},\tau}}])^{-1}\sE[Y^{\psi^{{\theta},\tau}}])$(cf.~\eqref{rewrTheta}). 
Then  by \eqref{reg_ls_discrete},
we can derive the following analogues of  
\eqref{eq:est}:
\begin{equation*}
\begin{split}
&\|\hat{\theta}-\theta^\star\|_2
\\
& = 
\|(V^{\psi^{{\theta,\tau}},\tau,m}+\tfrac{1}{m}I)^{-1}Y^{\psi^{{\theta,\tau}},\tau,m}
-(\sE[V^{\psi^{{\theta,\tau}}}])^{-1}\sE[Y^{\psi^{{\theta,\tau}}}]\|_2
\\
&\le 
\|(V^{\psi^{{\theta,\tau}},\tau,m}+\tfrac{1}{m}I)^{-1}-(\sE[V^{\psi^{{\theta,\tau}}}])^{-1}\|_2\|Y^{\psi^{{\theta},\tau},m,\tau}\|_2
+\|(\sE[V^{\psi^{{\theta,\tau}}}])^{-1}\|_2\|Y^{\psi^{{\theta,\tau},\tau},m}-\sE[Y^{\psi^{{\theta,\tau}}}]\|_2
\\
&\le
\|(V^{\psi^{{\theta,\tau}},\tau,m}+\tfrac{1}{m}I)^{-1}\|\|(\sE[V^{\psi^{{\theta,\tau}}}])^{-1}\|_2
\|Y^{\psi^{{\theta},\tau},m,\tau}\|_2
\|V^{\psi^{{\theta,\tau}},\tau,m}-\sE[V^{\psi^{{\theta,\tau}}}]+\tfrac{1}{m}I\|_2
\\
&\q 
+\|(\sE[V^{\psi^{{\theta,\tau}}}])^{-1}\|_2\|Y^{\psi^{{\theta,\tau},\tau},m}-\sE[Y^{\psi^{{\theta,\tau}}}]\|_2,
\end{split}
\end{equation*}
where $V^{\psi^{{\theta,\tau}},\tau,m}$ and 
$Y^{\psi^{{\theta,\tau}},\tau,m}$ are defined in 
\eqref{eq:rv_ls_psi_discrete}.
Note that 
\begin{align*}
\|V^{\psi^{{\theta,\tau}},\tau,m}-\sE[V^{\psi^{{\theta,\tau}}}]+\tfrac{1}{m}I\|_2
&\le
\|V^{\psi^{{\theta,\tau}},\tau,m}-
\sE[V^{\psi^{{\theta,\tau}},\tau}]\|_2
+\|\sE[V^{\psi^{{\theta,\tau}},\tau}]-
\sE[V^{\psi^{{\theta,\tau}}}]\|_2+\tfrac{1}{m},
\\
\|Y^{\psi^{{\theta,\tau},\tau},m}-\sE[Y^{\psi^{{\theta,\tau}}}]\|_2
&\le 
\|Y^{\psi^{{\theta,\tau},\tau},m}-
\sE[Y^{\psi^{{\theta,\tau},\tau}}]\|_2
+\|\sE[Y^{\psi^{{\theta,\tau},\tau}}]-
\sE[Y^{\psi^{{\theta,\tau}}}]\|_2,
\end{align*}
where for both inequalities,
the first term on the right-hand side
can be estimated by
Theorem \ref{thm:concentration_ls}
(uniformly in $N$),
and the second term
is of the magnitude $\cO(N^{-1})$
due to Lemma \ref{lemma:discrete_expectation}.
Hence, 
proceeding along the lines of the proof of 
Proposition \ref{theta_conc}
leads to  the desired result.
\end{proof}

\paragraph{Step 3: Proof of Theorem \ref{thm:regret_discrete}.}

The proof follows from  similar arguments as that of Theorem  \ref{thm:regret_continuous},
and we only present the main steps here.
As 
 (H.\ref{assum:ls}\ref{assum:non_degenerate}) holds with $\theta_0$ and 
 $\theta^\star$,
we can obtain from 
Propositions \ref{prop:non_degenerate}
and \ref{theta_conc_discrete}
 that,
 there exists
 a bounded set $\Phi_\eps\subset\sR^{(n+d)\t n}$ 
 and constants 
  $\bar{C}_1,\bar{C}_2\ge 1$, $n_0\in\sN$
  that 
for all $\theta\in  \Phi_\eps$,
$N\in \sN\cap [n_0,\infty)$
 and
 $\delta'\in (0,1/2)$, if 
$m\geq \bar{C}_1(-\ln\delta )$, 
 then   with probability at least $1-2\delta'$,
\bb 
|\hat{\theta}-\theta^\star|\leq \bar{C}_2\bigg(\sqrt{\frac{-\ln\delta'}{m}}+\frac{-\ln\delta' }{m}+\frac{(-\ln\delta')^{2}}{m^2}
+\frac{1}{N}
\bigg),
\ee
where 
$\hat{\theta}$ denotes the  right-hand side of \eqref{reg_ls_discrete}
 with the control $\psi^{\theta,\tau}$
 and stepsize $\tau=T/N$.
Then by proceeding along the lines of the proof of Theorem  \ref{thm:regret_continuous},
 there exists  $C_0>0$ and $n_0\in \sN$, such that for any given $\delta\in (0,\frac{3}{\pi^2})$,
if 
$m_0= C(-\ln \delta)$ with $C\ge C_0$ and $N_\ell\ge n_0$ for all $\ell\in \sN\cup\{0\}$,
then 
with   probability  at least $ 1-\frac{\pi^2}{3}\delta$,
\bb\label{eq:theta_ell_discrete}
|{\theta}_{\ell+1}-\theta^\star|\leq
 \bar{C}_2\bigg(\sqrt{\frac{-\ln\delta_\ell}{m_\ell}}+\frac{(-\ln\delta_\ell)}{m_\ell}+\frac{(-\ln\delta_\ell)^{2}}{m_\ell^2}+\frac{1}{N_\ell}\bigg),
\q  \fa \ell\in \sN\cup\{0\},
 \ee
 where $\delta_\ell=\delta/(\ell+1)^2$ and $m_\ell=2^\ell m_0$ for all $\ell$.
Consequently, we can conclude the desired regret  bound from Proposition 
\ref{prop:performance_gap_discrete}
(cf.~\eqref{eq:R_N_conts}),
with an additional term 
$\sum_{\ell=0}^{\ln M}m_\ell N_\ell^{-2}$ 
due to the time discretization errors in \eqref{item:performance_gap_discrete}
and \eqref{eq:theta_ell_discrete}.

\bibliography{ref_LQ_RL.bib}

\end{document}